\documentclass[12pt]{article}
\usepackage{amssymb,amsthm,amsmath,latexsym}
\usepackage{url}
\newtheorem{thm}{Theorem}[section]
\newtheorem{prop}[thm]{Proposition}
\newtheorem{lem}[thm]{Lemma}

\theoremstyle{remark}
\newtheorem{rem}[thm]{Remark}
\newcommand{\FF}{\mathbb{F}}
\newcommand{\ZZ}{\mathbb{Z}}
\newcommand{\0}{\mathbf{0}}
\newcommand{\1}{\mathbf{1}}

\newcommand{\cC}{\mathcal{C}}
\newcommand{\cD}{\mathcal{D}}

\DeclareMathOperator{\rank}{rank}

\begin{document}
\title{Characterization and classification of optimal LCD codes}

\author{
Makoto Araya\thanks{Department of Computer Science,
Shizuoka University,
Hamamatsu 432--8011, Japan.
email: {\tt araya@inf.shizuoka.ac.jp}},
Masaaki Harada\thanks{
Research Center for Pure and Applied Mathematics,
Graduate School of Information Sciences,
Tohoku University, Sendai 980--8579, Japan.
email: {\tt mharada@tohoku.ac.jp}.}
and 
Ken Saito\thanks{
Research Center for Pure and Applied Mathematics,
Graduate School of Information Sciences,
Tohoku University, Sendai 980--8579, Japan.
email: {\tt kensaito@ims.is.tohoku.ac.jp}.}
}


\maketitle

\begin{abstract}
Linear complementary dual (LCD) codes are linear codes that intersect
with their dual trivially.
We give a characterization of LCD codes
over $\FF_q$ having large minimum weights for $q \in \{2,3\}$.
Using the characterization, for arbitrary $n$,
we determine the largest minimum weights
 among LCD $[n,k]$ codes over $\FF_q$, where
 $(q,k) \in \{(2,4), (3,2),(3,3)\}$.
Moreover,  for arbitrary $n$, we give a complete classification of optimal 
 LCD $[n,k]$ codes over $\FF_q$, where
 $(q,k) \in \{(2,3), (2,4), (3,2),(3,3)\}$.
\end{abstract}

\section{Introduction}\label{sec:1}

Linear complementary dual (LCD for short) 
codes are linear codes that intersect with their dual
trivially.
LCD codes were introduced by Massey~\cite{Massey} and 
gave an optimum linear
coding solution for the two user binary adder channel.
Much work has been done concerning LCD codes
for both theoretical and practical reasons
(see e.g.~\cite{AH-C}, \cite{AH}, \cite{AHS},
\cite{CG},
\cite{CMTQ},
\cite{CMTQP},
\cite{FLFR},
\cite{GKLRW},
\cite{HS},
\cite{HS2},
\cite{LN},
\cite{LLGF},
\cite{Massey} and the references given therein).
In particular, we emphasize the recent work by 
Carlet, Mesnager, Tang, Qi and Pellikaan~\cite{CMTQP}.
It has been shown in~\cite{CMTQP} that
any code over $\FF_q$ is equivalent to some LCD code
for $q \ge 4$.
This motivates us to study LCD codes over $\FF_q$ for $q \in \{2,3\}$.
Codes over $\FF_2$ and $\FF_3$ are called binary codes and
ternary codes, respectively.

It is a fundamental problem to determine the largest minimum
weight $d_q(n,k)$ among all LCD $[n,k]$ codes over $\FF_q$ and to classify
LCD $[n,k,d_q(n,k)]$ codes for a given pair $(n,k)$.
Recently, 
much work has been done concerning this fundamental problem
(see e.g.~\cite{AH-C}, \cite{AH}, \cite{AHS},
\cite{CMTQ},
\cite{CMTQP},
\cite{FLFR},
\cite{GKLRW},
\cite{HS},
\cite{LLGF} and the references given therein).
A complete classification of all binary LCD codes of lengths up to
$13$ and all ternary LCD codes of lengths up to $10$ was given
in~\cite{AH-C} by using the mass formulas established in~\cite{CMTQ}.
An LCD $[n,k,d_q(n,k)]$ code over $\FF_q$ is called optimal.
For arbitrary $n$,
the largest minimum weights $d_2(n,2)$ and $d_2(n,3)$ were determined 
in~\cite{GKLRW} and~\cite{HS}, respectively.
A classification of binary optimal LCD codes was done
in~\cite{HS} for dimension $2$.
In this paper, by considering the simplex codes, 
we give a characterization of LCD codes
over $\FF_q$ having large minimum weights for $q \in \{2,3\}$.
Using the characterization, for arbitrary $n$,
we complete the determination of 
the largest minimum weights $d_q(n,k)$
and a classification of optimal 
LCD $[n,k]$ codes over $\FF_q$, 
where $k \le 6-q$ and $q \in \{2,3\}$.

The paper is organized as follows.
In Section~\ref{sec:pre}, 
we give some definitions, notations and basic results used in this paper.
In Section~\ref{sec:basic}, we give some basic results related to LCD
codes.  In particular, we
give an observation on LCD codes related to the simplex
codes (Lemma~\ref{lem:190125-5}).
In Section~\ref{sec:char}, we give a 
characterization of LCD codes with
large minimum weights (Theorem~\ref{thm:main}),
which is the main result of this paper.
Lemma~\ref{lem:190125-5} is a key idea for Theorem~\ref{thm:main}.
Theorem~\ref{thm:main} (i) claims that
there is a one-to-one correspondence between
equivalence classes of LCD $[n,k,d]$ codes over $\FF_q$
with dual distances $d^\perp \ge 2$
and
equivalence classes of LCD
$[q \cdot r_{q,n,k,d},k,(q-1)r_{q,n,k,d}]$ codes over $\FF_q$
with dual distances $d^\perp \ge 2$ for $k \ge k_0$,
under the assumption that
$(q,k_0) \in \{(2,3),(3,2)\}$, $qd-(q-1)n \ge 1$
and $q \cdot r_{q,n,k,d} \ge k$, where $r_{q,n,k,d}=q^{k-1}n-\frac{q^k-1}{q-1}d$.
As a consequence, it follows that
if there is no LCD $[q \cdot r_{q,n,k,d},k,(q-1)r_{q,n,k,d}]$ code
over $\mathbb{F}_q$ with dual distance $d^\perp \ge 2$,
then there is no LCD $[n,k,d]$ code over $\mathbb{F}_q$,
under the same assumption (Theorem~\ref{thm:main} (ii)).
In addition, 
we modify Theorem~\ref{thm:main} to the form to be used easily
under some assumption
for our study 
in Sections~\ref{sec:2} and~\ref{sec:3}
 (Theorem~\ref{thm:main2}).
Roughly speaking,
Theorem~\ref{thm:main2} (i) says that a classification of
LCD codes with large minimum weights for small length
yields that of LCD codes with large minimum weights for
a family of lengths, and
Theorem~\ref{thm:main} (ii) says that the nonexistence of
LCD codes with large minimum weights for small length
yields that of LCD codes with large minimum weights for
a family of lengths.
In Sections~\ref{sec:2} and~\ref{sec:3},
by Theorem~\ref{thm:main2}, 
for arbitrary $n$,
we complete the determination of 
the largest minimum weights $d_q(n,k)$
and  a classification of optimal 
LCD $[n,k]$ codes over $\FF_q$, where $k \le 6-q$ and $q \in \{2,3\}$.
This result is mainly obtained by showing 
the nonexistence of
LCD codes with large minimum weights
and giving a classification of optimal
LCD codes for small lengths.
The nonexistence and the classification given
in Sections~\ref{sec:2} and~\ref{sec:3}
are obtained by computer calculations
by using the method given in 
Section~\ref{sec:Cmethod}.


\section{Preliminaries}\label{sec:pre}
In this section, 
we give some definitions, notations and basic results used in this
paper.

\subsection{Definitions and notations}

We denote the finite field of order $q$ by $\FF_q$, where $q$ is a prime power.
A (linear) $[n,k]$ {\em code} $C$ over $\FF_q$ is a $k$-dimensional subspace of $\FF_q^n$.
Codes over $\FF_2$ and $\FF_3$ are called {\em binary} codes and
{\em ternary} codes, respectively.
The parameters $n$ and $k$ are called the {\em length} and the
{\em dimension} of $C$, respectively.
A {\em generator matrix} of an $[n,k]$ code $C$ is a $k \times n$ matrix whose rows are basis of $C$. 
The {\it support} of a vector $x=(x_1,x_2,\ldots,x_n) \in \FF_q^n$ is the subset $\{i \mid x_i \neq 0\}$ of $\{1,2,\ldots,n\}$. 
The {\em weight} of a vector $x \in \FF_q^n$ is the cardinality of the support of $x$.
A vector of $C$ is called a {\em codeword} of $C$.
The minimum nonzero weight of all codewords in $C$ is called the {\em minimum weight} $d(C)$ of $C$.
An $[n,k]$ code with minimum weight $d$ is called an $[n,k,d]$ code.
Two $[n,k]$ codes $C_1$ and $C_2$ over $\FF_q$ are {\em equivalent},
denoted $C_1 \cong C_2$,
if there is an $n \times n$ monomial matrix $P$ over $\FF_q$ such that $C_2=\{cP \mid c \in C_1\}$.

The {\em dual} code $C^\perp$ of an $[n,k]$ code $C$ over $\FF_q$ is defined as $C^\perp=\{x \in \FF_q^n \mid x \cdot y = 0 \text{ for all } y \in C\}$, where $x \cdot y$ is the standard inner product.
A code $C$ is called {\em self-orthogonal} if $C \subset C^\perp$.
A code $C$ is called a {\em linear complementary dual}
(LCD for short) 
code if $C \cap C^\perp = \{\0_n\}$,
where $\0_n$ denotes the zero vector of length $n$.
The following characterization is due to Massey~\cite{Massey}.
Throughout this paper, we use the following proposition
without mentioning this, 
when we determine whether a given code is LCD or not.

\begin{prop}\label{prop:lcd}
Let $G$ be a generator matrix of a code $C$ over $\FF_q$.
 Then $C$ is LCD if and only if $GG^T$ is nonsingular,
 where $G^T$ denotes the transpose of $G$.
\end{prop}

We use the following notations throughout this paper.
Let $\ZZ_{\ge 0}$ denote the set of nonnegative integers.
Let $I_k$ denote the identity matrix of order $k$,
and let $O$ denote the zero matrix of appropriate size.
For a positive integer $s$ and a matrix $A$, let $A^{(s)}$ denote
the juxtaposition
$\begin{pmatrix}
&  A &A&\cdots &A&
\end{pmatrix}$
of $s$-copies of $A$.
Let $\1_n$ denote the all-one vector of length $n$.
For $q$, $n$, $k$ and $d$, we use the following notations:
\begin{align}
[k]_q&=\frac{q^k-1}{q-1}, \label{eq:k}\\
r_{q,n,k,d}&=q^{k-1}n-[k]_q \cdot d, \label{eq:r}\\
g_q(n,k)&=\max\left\{d \in \mathbb{Z}_{\ge 0} ~\middle|~ n \ge
 \sum_{i=0}^{k-1} \left\lceil \frac{d}{q^i}\right\rceil\right\},
\label{eq:g}
\\
d_q(n,k)&=\max\left\{d \in \mathbb{Z}_{\ge 0} \mid
\text{ there is an LCD $[n,k,d]$ code over $\FF_q$}
\right\}. 
\label{eq:dmax}
\end{align}
By the Griesmer bound, any $[n,k]$ code over $\mathbb{F}_q$ has minimum weight at most $g_q(n,k)$.
An LCD $[n,k,d_q(n,k)]$ code over $\FF_q$ is called {\em optimal}.
The minimum weight of the dual code $C^\perp$ of a code $C$ is called
the {\em dual distance} of $C$ and it is denoted by $d^\perp$.
Let $N_q(n,k)$ denote the number of all inequivalent optimal LCD
$[n,k]$ codes over $\FF_q$ with dual distances $d^\perp \ge 2$.

\subsection{Some observations on generator matrices}

\begin{lem}\label{lem:GM}
Let $C$ and $C'$ be $[n,k]$ codes over $\FF_q$.
Let $M$ and $M'$ be generator matrices of $C$ and $C'$, respectively.
If $C \cong C'$, then
there is a nonsingular matrix $U$ and there is a monomial
matrix $P$ such that $M=UM'P$.
\end{lem}
\begin{proof}
 It is trivial.
\end{proof}
%

Although the following lemma is somewhat trivial, 
we give a proof for the sake of completeness.

\begin{lem}\label{lem:1}
Let $A$ be a $k \times \ell$ $\mathbb{F}_q$-matrix.
Let $B$ and $B'$ be $k \times m$ $\mathbb{F}_q$-matrices.
If there is an $(\ell+m) \times (\ell+m)$ monomial matrix $P$ such that
$
\begin{pmatrix}
& A & B' &
\end{pmatrix}
=
\begin{pmatrix}
& A & B &
\end{pmatrix}
P$,
then there is an $m \times m$ monomial matrix $P^*$ such that
$B'=BP^*$.
\end{lem}
\begin{proof}
The proof is by induction on $\ell$.
It is sufficient to show the case $\ell=1$.
Let $a$ be a vector of $\mathbb{F}_q^k$.
Suppose that there is an $(m+1) \times (m+1)$ monomial matrix $P=(p_{i,j})$ such that
\begin{equation}\label{eq:11}
\begin{pmatrix}
& a^T & B' &
\end{pmatrix}
=
\begin{pmatrix}
& a^T & B &
\end{pmatrix}
P.
\end{equation}

Suppose that $p_{1,1} \ne 0$.
Let $P'$ be the $m \times m$ monomial matrix obtained from $P$
by deleting the first row and the first column.
Since $p_{1,j}=0$ for each $j \in \{2,3,\ldots,m+1\}$,
we have $B'=BP'$.

Now suppose that $p_{1,1}=0$.
Then $p_{1,j_0}\ne 0$ for some $j_0 \in \{2,3,\ldots,m+1\}$.
From~\eqref{eq:11}, the $(j_0-1)$-th column of $B'$ is equal to $p_{1,j_0}a^T$.
Let $Q=(q_{i,j})$ be the $(m+1) \times (m+1)$ monomial matrix such that
$q_{i,i}=1$ for each $i \in \{2,3,\ldots,m+1\} \setminus \{j_0\}$, and
$(q_{1,j_0},q_{j_0,1})=(p_{1,j_0},p_{1,j_0}^{-1})$.
Then we have
\begin{equation}\label{eq:2}
\begin{pmatrix}
& a^T & B' &
\end{pmatrix}
Q=
\begin{pmatrix}
& a^T & B' &
\end{pmatrix}
\text{ and }
PQ=
\begin{pmatrix}
1 & 0 & \cdots & 0 \\
0 &&& \\
\vdots &&P^*& \\
0 &&& \\
\end{pmatrix},
\end{equation}
for some $m \times m$ monomial matrix $P^*$,
where the blanks are filled up with zero's.
From~\eqref{eq:11} and~\eqref{eq:2}, we have
\begin{align*}
\begin{pmatrix}
& a^T & B' &
\end{pmatrix}
=
\begin{pmatrix}
& a^T & B' &
\end{pmatrix}
Q 
=
\begin{pmatrix}
& a^T & B &
\end{pmatrix}
PQ 
=
\begin{pmatrix}
& a^T & BP^* &
\end{pmatrix}.
\end{align*}
Hence, we have $B'=BP^*$, which gives the desired conclusion.
\end{proof}

Throughout this paper,
we remark that the blanks are filled up with zero's
for each matrix.

%

\section{Basic properties on LCD codes}
\label{sec:basic}

In this section, we 
give some basic results related to LCD codes.  In particular, we
give an observation on LCD codes related to the simplex
codes (Lemma~\ref{lem:190125-5}).
This is a key idea for our characterization of LCD codes with
large minimum weights (Theorem~\ref{thm:main}),
which is the main result of this paper.
Throughout this section, we assume that $q \in \{2,3\}$.


\subsection{Equivalence classes}

Let $C$ be an LCD $[n,k,d]$ code over $\FF_q$.
Define an $[n+1,k,d]$ code $\overline{C}$ over $\FF_q$ as
$\overline{C}=\{(x,0) \mid x \in C\}$.
It is trivial that $\overline{C}$ is an LCD  code with dual distance $1$.

\begin{lem}[{\cite[Proposition~3]{AH-C}}]\label{lem:class}
Let $\cC^q_{n,k,d}$ denote all equivalence classes of LCD $[n,k,d]$
codes over $\FF_q$.
Let $\cD^q_{n,k,d}$ denote all equivalence classes of LCD $[n,k,d]$
codes over $\FF_q$ with dual distances $d^\perp\ge 2$.
Let $\overline{\cC^q_{n-1,k,d}}$ denote 
all equivalence classes containing $\overline{C_1},\overline{C_2},\ldots,
\overline{C_t}$,
where
$C_1,C_2,\ldots,C_t$ denote representatives of
$\cC^q_{n-1,k,d}$, and
$t=|\overline{\cC^q_{n-1,k,d}}|$.
Then 
$\cC^q_{n,k,d}= \cD^q_{n,k,d} \cup \overline{\cC^q_{n-1,k,d}}$.
\end{lem}

For a classification of LCD $[n,k,d]$ codes over $\FF_q$,
by the above lemma, it is sufficient to consider
a classification of LCD $[n,k,d]$ codes over $\FF_q$
with dual distances $d^\perp \ge 2$.
In addition,
if there is an LCD $[n,k,d]$ code $C$ over $\mathbb{F}_q$ and
$d_q(n-1,k) \le d-1$, then $d(C^\perp) \ge 2$
(see~\eqref{eq:dmax} for the notation $d_q(n,k)$). 

\subsection{Relation between LCD codes with $d^\perp \ge 2$ and 
LCD codes with $d^\perp =1$}

\begin{lem}\label{lem:a}
Suppose that $q \in \{2,3\}$ and $k \ge 2$.
If there is an LCD $[n,k,d_0]$ code over $\FF_q$ with dual
distance $d^\perp \ge 2$,
then there is an LCD $[n+1,k,d]$ code over $\FF_q$ with
$d \in \{d_0, d_0+1\}$ and dual distance $d^\perp \ge 2$.
\end{lem}
\begin{proof}
 Suppose that there is an LCD $[n,k,d_0]$ code $C_0$ over $\FF_q$ with
 dual distance $d(C_0^\perp) \ge 2$.

Suppose that $q=2$.
Recall that a binary code is called {\em even} if the weights
of all codewords are even and 
a binary code which is not even is called {\em odd}.
If $C_0$ is odd,
then there is a generator matrix $G_0$ of $C_0$
such that $G_0G_0^T=I_k$~\cite[Theorem~3]{CMTQ}.
If $C_0$ is even,
then $k$ is even and there is a generator matrix $G'_0$ of $C_0$ such that
\[
G'_0{G'_0}^T=
\begin{pmatrix}
J_2 && \\
& \ddots & \\
&& J_2
\end{pmatrix},
\]
where
$J_2=
\begin{pmatrix}
0 & 1 \\
1 & 0
\end{pmatrix}
$~\cite[Theorem~5]{CMTQ}.

Suppose that $q=3$.
From the proof of~\cite[Proposition~4]{HS2},
$C_0$ has generator matrix of form
$
G_0=
\left(\begin{array}{ccc}
&x& \\
&G_1& \\
\end{array}\right)$,
satisfying  that
$x \cdot x \neq 0$ and $xG_1^T=\0_{k-1}$.
The code $C_1$ with generator matrix $G_1$ is a ternary LCD $[n,k-1]$ code.

\begin{itemize}
\item $k=2$:
$G_1$ is regarded as a codeword $x^* \in C_1$ with $x^* \cdot x^* \neq 0$,
since $C_1$ is LCD.

\item $k \ge 3$: Since $C_1$ is LCD,
$C_1$ has generator matrix of form
$
G'_1=
\left(\begin{array}{ccc}
&x'& \\
&G_2'& \\
\end{array}\right)$,
satisfying that
$x' \cdot x' \neq 0$ and $x'{G_2'}^T=\0_{k-2}$.
Since both matrices $G_1$ and $G'_1$ are generator matrices of $C_1$,
\[
G'_0=
\left(\begin{array}{ccc}
&x& \\
&G'_1& \\
\end{array}\right)
=
\left(\begin{array}{ccc}
&x& \\
&x'& \\
&G'_2& \\
\end{array}\right)
\]
is also a generator matrix of $C_0$.
Since $x\in C_1^\perp$ and $G'_1$ is a generator matrix of $C_1$,
we have $x \cdot x'=0$ and $x{G'_2}^T=\0_{k-2}$.
\end{itemize}

Finally, consider the following $k \times (n+1)$ $\FF_q$-matrix $G$:
\[G=
\begin{cases}
\begin{pmatrix}
&G_0 & h_2^T&
\end{pmatrix}
 & \text{if } q=2 \text{ and } C_0 \text{ is odd}, \\
\begin{pmatrix}
&G'_0 & h_2^T&
\end{pmatrix}
 & \text{if } q=2 \text{ and } C_0 \text{ is even}, \\
\begin{pmatrix}
&G_0 & h_3^T&
\end{pmatrix}
 & \text{if } q=3 \text{ and } k=2, \\
\begin{pmatrix}
&G'_0 & h_3^T&
\end{pmatrix}
 & \text{if } q=3 \text{ and } k \ge 3, \\
\end{cases}
\]
where
\begin{align*}
h_2&=
\begin{cases}
(1,1) & \text{if } k=2, \\
(1,1,\0_{k-2}) & \text{if } k \ge 3, \\
\end{cases}
\\
h_3&=
\begin{cases}
(1,0) & \text{if } k=2 \text{ and } x \cdot x=1, \\
(0,1) & \text{if } k=2, x \cdot x=2 \text{ and } x^* \cdot x^*=1, \\
(1,1) & \text{if } k=2 \text{ and } x \cdot x= x^* \cdot x^*=2, \\
(1,0,\0_{k-2}) & \text{if } k \ge 3 \text{ and } x \cdot x=1, \\
(0,1,\0_{k-2}) & \text{if } k \ge 3, x \cdot x=2 \text{ and } x' \cdot x'=1, \\
(1,1,\0_{k-2}) & \text{if } k \ge 3 \text{ and } x \cdot x=x' \cdot x'=2.
\end{cases}
\end{align*}
It is not difficult to show that $GG^T$ is nonsingular.
Hence, the code with generator matrix $G$ is an LCD $[n+1,k,d]$ code over $\FF_q$ with
$d \in \{d_0, d_0+1\}$ and dual distance $d^\perp \ge 2$.
\end{proof}

\begin{prop}\label{prop:dd1}
Suppose that $q \in \{2,3\}$, $k \ge 2$ and $d_0$ is a positive integer.
If there is no LCD $[n,k,d]$ code over $\FF_q$ with
$d \ge d_0$ and dual distance $d^\perp \ge 2$,
then there is no LCD $[n,k,d]$ code over $\FF_q$ with
$d \ge d_0$ and dual distance $1$.
\end{prop}
\begin{proof}
Suppose that there is an LCD $[n,k,d]$ code $C$ over $\FF_q$ with
$d \ge d_0$ and dual distance $1$.
By deleting all zero columns of a generator matrix of $C$,
an LCD $[n-\ell,k,d]$ code over $\FF_q$ with dual distance $d^\perp \ge 2$
is constructed for some positive integer $\ell$.
Then an LCD $[n,k,d]$ code over $\FF_q$
with $d \ge d_0$ and dual distance $d^\perp \ge 2$ is constructed
by Lemma~\ref{lem:a}.
\end{proof}

\subsection{LCD codes related to the simplex codes}

In this subsection, we suppose that $(q,k_0) \in \{(2,3),(3,2)\}$.
We define the $k \times [k]_q$ $\mathbb{F}_q$-matrices $S_{q,k}$
by inductive constructions as follows:
\begin{align*}
S_{q,1}&=
\begin{pmatrix}
1
\end{pmatrix}, \\
S_{2,k}&=
\left(
\begin{array}{ccc}
S_{2,k-1} & \0_{k-1}^T & S_{2,k-1} \\ 
\0_{[k-1]_2} & 1 & \1_{[k-1]_2}
\end{array}
\right),  \\
S_{3,k}&=
\left(
\begin{array}{cccc}
S_{3,k-1} & \0_{k-1}^T & S_{3,k-1} & S_{3,k-1} \\ 
\0_{[k-1]_3} & 1 & \1_{[k-1]_3} & 2\1_{[k-1]_3}
\end{array}
\right),
\end{align*}
for $k \ge 2$ (see~\eqref{eq:k} for the notation $[k]_q$).
The matrix $S_{q,k}$ is a generator matrix of the simplex
$[[k]_q,k,q^{k-1}]$ code over $\FF_q$.
The simplex $[[k]_q,k,q^{k-1}]$ code is a constant weight code~\cite[Theorem~2.7.5]{HP}.
Moreover, 
the simplex $[[k]_q,k,q^{k-1}]$ code is self-orthogonal if $k \ge k_0$~\cite[Theorems~1.4.8~(ii) and 1.4.10~(i)]{HP}.

Let $h_{q,k,i}$ be the $i$-th column of the matrix $S_{q,k}$.
For a vector $m=(m_1,m_2,\ldots,m_{[k]_q}) \in \mathbb{Z}_{\ge 0}^{[k]_q}$, we define a $k \times \sum_{i=1}^{[k]_q}m_i$ $\mathbb{F}_q$-matrix $G_{q,k}(m)$, which consists of $m_i$ columns $h_{q,k,i}$ for each $i$ as follows:
\begin{equation}\label{eq:Gqkm}
G_{q,k}(m)=
\left(
 h_{q,k,1} \cdots h_{q,k,1} h_{q,k,2} \cdots h_{q,k,2}
\cdots h_{q,k,[k]_q} \cdots h_{q,k,[k]_q}
\right).
\end{equation}
Here we remark that $m_i=0$ means no column of $G_{q,k}(m)$ is $h_{q,k,i}$.
Throughout this paper,
we denote by $C_{q,k}(m)$ the code with generator matrix $G_{q,k}(m)$.
For any $[n,k]$ code $C$ over $\mathbb{F}_q$ with
dual distance $d^\perp \ge 2$,
there is a vector $m=(m_1,m_2,\ldots,m_{[k]_q}) \in
\mathbb{Z}_{\ge 0}^{[k]_q}$
such that $C \cong C_{q,k}(m)$ and
$n=\sum_{i=1}^{[k]_q}m_i$.


It is well known that 
any $[[k]_q\cdot s,k,q^{k-1}s]$ code over $\mathbb{F}_q$ is
equivalent to $C_{q,k}(s\1_{[k]_q})$
if $(q,k_0) \in \{(2,3),(3,2)\}$, 
$k \ge k_0$ and $s$ is a positive integer~\cite{Bonisoli}
(see also~\cite{Maruta}).
Hence, we have the following:

\begin{lem}\label{lem:n0}
If $(q,k_0) \in \{(2,3),(3,2)\}$, 
$k \ge k_0$ and $n \equiv 0 \pmod{[k]_q}$, then
there is no LCD $[n,k,g_q(n,k)]$ code over $\mathbb{F}_q$.
\end{lem}


\begin{lem}\label{lem:190228-1}
Suppose that $(q,k_0) \in \{(2,3),(3,2)\}$, 
$k \ge k_0$, $s$ is a positive integer and 
$m=(m_1,m_2,\ldots,m_{[k]_q}) \in \mathbb{Z}_{\ge 0}^{[k]_q}$.
If $C_{q,k}(m)$ is an LCD
$[\sum_{i=1}^{[k]_q}m_i,k,d_0]$
code over $\FF_q$, then the code $C$ with generator matrix
\[
\left(\begin{array}{cc}
 S_{q,k}^{(s)} &G_{q,k}(m)
       \end{array}\right)
\]
is an LCD $[n,k,d]$ code, where
$
n=\sum_{i=1}^{[k]_q}m_i+[k]_q \cdot s$  and  $d=d_0+q^{k-1}s$.
\end{lem}
\begin{proof}
Let $G_1$ (resp.\ $G_2$) be a generator matrix of 
a self-orthogonal $[n_1,k,d_1]$ code (resp.\ an LCD $[n_2,k,d_2]$ code)
over $\mathbb{F}_q$.
It is trivial that the code with generator matrix
$
\begin{pmatrix}
& G_1 & G_2& 
\end{pmatrix}
$ is an LCD $[n_1+n_2,k]$ code with minimum weight at least $d_1+d_2$.

The code $D$ with generator matrix $S_{q,k}^{(s)}$
is a self-orthogonal $[[k]_q\cdot s,k,q^{k-1}s]$ code.
From the assumption, $C_{q,k}(m)$ is an LCD
$[\sum_{i=1}^{[k]_q}m_i,k,d_0]$ code.
By the above argument, $C$ is 
an LCD $[\sum_{i=1}^{[k]_q}m_i+[k]_q \cdot s,k]$
code with minimum weight at least $d_0+q^{k-1}s$.
Since $D$ is a constant  weight code,
there is a codeword of weight $d_0+q^{k-1}s$ in $C$.
The result follows.
\end{proof}



Now we give definitions concerning symmetric designs.
A $2$-$(v,\ell,\lambda)$ {\em design} $\cD$ is a set $X$ of $v$ {\em points}
together
with a collection of $\ell$-subsets of $X$ called {\em blocks}
such that every $2$-subset of $X$ is contained in exactly $\lambda$ blocks.
The number of blocks that contain a given
point is traditionally denoted by $r$, and the total number of
blocks is denoted by $b$.
The $b \times v$ {\em incidence matrix} $A=(a_{i,j})$ of $\cD$ is defined by
$a_{i,j}=1$ if the $j$-th point is contained in the $i$-th block
and $a_{i,j}=0$ otherwise.
A $2$-$(v,\ell,\lambda)$ design $\cD$ is called {\em symmetric} if $v=b$.
It is known that $\cD$ is symmetric if and only if $r=\ell$.
For undefined terms concerning designs, see e.g.~\cite{CL}.

Suppose that there is a symmetric $2$-$(v,\ell,\lambda)$ design $\cD$.
Let $A$ be the incidence matrix of $\cD$.
Let $m=(m_1,m_2,\ldots,m_{v}) $ be a vector of  $\mathbb{Z}_{\ge 0}^{v}$.
If each entry of the $v \times 1$ $\ZZ$-matrix $A m^T$ is 
at least $d$,
then
\begin{equation}\label{eq:1}
\frac{\ell d-\lambda \sum_{j=1}^v m_j}{\ell-\lambda} \le m_i \le
\sum_{j=1}^v m_j-\frac{v-\ell}{\ell-\lambda}d,
\end{equation}
for each $i \in \{1,2,\ldots,v\}$~\cite[Lemma~3.1]{AHS}.
The following lemma is a key idea for
the determination of possible vectors $m$ for
the codes $C_{q,k}(m)$
with generator matrices $G_{q,k}(m)$ of form~\eqref{eq:Gqkm}.

\begin{lem}\label{lem:190125-5}
Suppose that $(q,k_0) \in \{(2,3),(3,2)\}$ and $k \ge k_0$.
If the code $C_{q,k}(m)$ has minimum weight at least $d$, then
\begin{equation}\label{eq:mi}
 qd-(q-1)n \le m_i \le n-\frac{q^{k-1}-1}{(q-1)q^{k-2}}d,
\end{equation}
for each $i \in \{1,2,\ldots,[k]_q\}$, where
 $m=(m_1,m_2,\ldots,m_{[k]_q}) \in \mathbb{Z}_{\ge 0}^{[k]_q}$
 and $n=\sum_{i=1}^{[k]_q}m_i$.
\end{lem}
\begin{proof}
By the Assmus--Mattson theorem, the supports of all nonzero codewords
in the simplex $[[k]_q,k,q^{k-1}]$ code over $\FF_q$ yield a
symmetric $2$-$([k]_q,q^{k-1},(q-1)q^{k-2})$ design $\cD_{q,k}$ if $k \ge k_0$.
Let $A$ be the incidence matrix of $\cD_{q,k}$.
Note that the simplex $[[k]_q,k,q^{k-1}]$ code is a constant weight code.
Thus, by considering the construction of $C_{q,k}(m)$,
the weight of any nonzero codeword of $C_{q,k}(m)$
is written as one of the entries of $A(m_1,m_2,\ldots,m_v)^T$.
Since $C_{q,k}(m)$ has minimum weight at least $d$,
each entry of $A(m_1,m_2,\ldots,m_v)^T$ is at least $d$.
From~\eqref{eq:1}, we have the
desired conclusion.
\end{proof}

%


\begin{rem}
Recently, quaternary Hermitian LCD codes with large minimum weights have been 
studied by considering simplex codes in~\cite{AHS} and~\cite{LLGF}.
Lemma~\ref{lem:190228-1}
is an $\FF_q$-analogy ($q\in\{2,3\}$) of~\cite[Lemmas~3.2 and 3.3]{LLGF}
(see also~\cite{FLFR} and~\cite{LN} for $q=2$).
Lemma~\ref{lem:190125-5}
is an $\FF_q$-analogy ($q\in\{2,3\}$) of~\cite[Lemma~3.2]{AHS}.
Lemma~\ref{lem:char} is also an
$\FF_q$-analogy ($q\in\{2,3\}$) of~\cite[Theorem~3.4]{AHS}.
\end{rem}

\section{Characterization of LCD codes}
\label{sec:char}

Throughout this section, 
we assume that $(q,k_0) \in \{(2,3),(3,2)\}$ and $qd-(q-1)n \ge 1$.
In this section, we give a 
characterization of LCD codes with
large minimum weights (Theorem~\ref{thm:main}),
which is the main result of this paper.
In addition, 
we modify Theorem~\ref{thm:main} to the form to be used easily
under some assumption~\eqref{eq:as} for our study 
in Sections~\ref{sec:2} and~\ref{sec:3}
 (Theorem~\ref{thm:main2}).

\subsection{Theorem~\ref{thm:main} and its proof}



%
%

\begin{lem}\label{lem:char}
Suppose that $(q,k_0) \in \{(2,3),(3,2)\}$,
$k \ge k_0$ and $qd-(q-1)n \ge 1$.
Let $r_{q,n,k,d}$ denote the integer defined in~\eqref{eq:r}.
\begin{enumerate}
\item If $q \cdot r_{q,n,k,d} < k$,
then there is no LCD $[n,k,d]$ code over $\mathbb{F}_q$
with dual distance $d^\perp \ge 2$.
\item
If there is an LCD $[n,k,d]$ code $C$ over $\FF_q$ with
dual distance $d(C^\perp) \ge 2$,
then there is an LCD $[n,k,d]$ code $C'$ such that
$C \cong C'$ and 
$C'$ has generator matrix of the following form:  
\begin{equation}\label{eq:G}
G=
\begin{pmatrix}
& S_{q,k}^{(qd-(q-1)n)} &  G_0 &
\end{pmatrix},
\end{equation}
where 
$G_0$ is a generator matrix of some LCD
$[q \cdot r_{q,n,k,d},k,(q-1)r_{q,n,k,d}]$ code.
 \end{enumerate}
\end{lem}
\begin{proof}
Since $C$ is an $[n,k]$ code over $\mathbb{F}_q$
with dual distance $d(C^\perp) \ge 2$,
there is a vector $m=(m_1,m_2,\ldots,m_{[k]_q}) \in
\mathbb{Z}_{\ge 0}^{[k]_q}$
such that $C \cong C_{q,k}(m)$ and
$n=\sum_{i=1}^{[k]_q}m_i$.
Since the minimum weight of $C$ is $d$, we have
\[qd-(q-1)n \le m_i,\]
by Lemma~\ref{lem:190125-5}.
Since the generator matrix $G_{q,k}(m)$ of $C_{q,k}(m)$
consists of at least $qd-(q-1)n$ columns
$h_{q,k,i}$ for each $i \in\{1,2,\ldots,[k]_q\}$, we
obtain a matrix $G$ of form~\eqref{eq:G},
 by permuting columns of $G_{q,k}(m)$.
Here $G_0$ is a $k \times n_0$ matrix, where
\[
n_0= n-(qd-(q-1)n)[k]_q
= q(q^{k-1}n-[k]_q \cdot d)=q \cdot r_{q,n,k,d}. 
\]
Since $S_{q,k}S_{q,k}^T=O$, we have $GG^T=G_0G_0^T$.
Since $C$ is LCD,  we have
\begin{equation}\label{eq:rank}
n_0 \ge \rank(G_0) \ge \rank(G_0G_0^T)=\rank(GG^T)=k.
\end{equation}
This proves the assertion (i).

Let $C_0$ be the code with generator matrix $G_0$.
Again, \eqref{eq:rank} shows that 
$C_0$ is an LCD code of dimension $k$.
It remains to show that 
the minimum weight  $d_0$ of $C_0$ is $(q-1)r_{q,n,k,d}$.
The code $C'$ with generator matrix $S_{q,k}^{(qd-(q-1)n)}$ is
a self-orthogonal $[n',k,d']$ code, where
\[
n'=(qd-(q-1)n)[k]_q \text{ and }
d'=(qd-(q-1)n)q^{k-1}.
\]
By Lemma~\ref{lem:190228-1}, 
we have 
\[
d=d_0+d' \text{ and }
d_0=(q-1)\left(q^{k-1}n-[k]_q \cdot d\right)=(q-1)r_{q,n,k,d}.\]
This completes the proof.
\end{proof}

\begin{rem}
If $qd-(q-1)n \ge 1$, then we have
\begin{align*}
(q-1)(n-q \cdot r_{q,n,k,d})
&\ge -(q-1)n(q^k-1)+((q-1)n+1)(q^k-1) \\
&=q^k-1.
\end{align*}
     This means that the length of $C_0$ is less than
     the length of $C$.
\end{rem}

Let $C$ and $C'$ be LCD $[n,k,d]$ codes over $\FF_q$ with
dual distances $d(C^\perp) \ge 2$ and $d({C'}^\perp) \ge 2$.
By Lemma~\ref{lem:char},
there are LCD $[n,k,d]$ codes $D$ and $D'$ over $\FF_q$ satisfying
the following conditions:
\begin{enumerate}
\item $C \cong D$ and $C' \cong D'$,
\item $D$ and $D'$ have generator matrices
\[
G=
\left(\begin{array}{cccc}
&S_{q,k}^{(qd-(q-1)n)} &  G_0 &
\end{array}\right) \text{ and }
G'=
\left(\begin{array}{ccccc}
&S_{q,k}^{(qd-(q-1)n)} &  G'_0 &
\end{array}\right),
\]
where $G_0$ and $G'_0$ are generator matrices of some LCD
$[q \cdot r_{q,n,k,d},k,(q-1)r_{q,n,k,d}]$ codes $C_0$ and $C'_0$,
respectively.
\end{enumerate}

\begin{lem}\label{lem:main1}
Let $C,C',C_0$ and $C'_0$ be the codes described as above.
If $C_0 \cong C'_0$, then $C \cong C'$.
\end{lem}
\begin{proof}
Since $C_0 \cong C'_0$,
there is a $k \times k$ nonsingular matrix $U$
and 
there is a $q \cdot r_{q,n,k,d} \times q \cdot r_{q,n,k,d}$ monomial
matrix $P$ 
such that $G_0=UG_0'P$ by Lemma~\ref{lem:GM}.
From the definition of $S_{q,k}$,  there is a
$[k]_q \times [k]_q$ monomial matrix $Q$ such that $S_{q,k}=(US_{q,k})Q$.
Thus, we have
\begin{align*}
\left(\begin{array}{cccc}
&S_{q,k}^{(qd-(q-1)n)} &  G_0 &
\end{array}\right)
&=
\left(\begin{array}{cccc}
&(US_{q,k}Q)^{(qd-(q-1)n)} & UG'_0P&
  \end{array}\right)
\\
&=
U\left(\begin{array}{cccc}
&S_{q,k}^{(qd-(q-1)n)} &  G'_0 &
\end{array}\right)
\begin{pmatrix}
Q    &&& \\
&\ddots&  &\\\
&      & Q & \\
& && P
\end{pmatrix}.
\end{align*}
It follows that $C \cong C'$.
\end{proof}



\begin{lem}\label{lem:main2}
Suppose that $q \in \{2,3\}$ and $s \in \ZZ_{\ge 0}$.
Let $B$ and $B'$ be $[\ell,k]$ codes over $\FF_q$ with generator matrices $M$ and $M'$, respectively. 
Let $C$ and $C'$ be $[n,k]$ codes over $\FF_q$ with generator matrices
$
\begin{pmatrix}
& S_{q,k}^{(s)} &  M &
\end{pmatrix}$
and 
$
\begin{pmatrix}
& S_{q,k}^{(s)} &  M' &
\end{pmatrix}$, respectively, where $n=[k]_q \cdot s+\ell$.
If $B \not\cong B'$, then $C \not\cong C'$.
\end{lem}
\begin{proof}
Suppose that $C \cong C'$.
By Lemma~\ref{lem:GM},
there is a $k \times k$ nonsingular matrix $U$ and
there is an $n \times n$ monomial matrix $P$ such that
\begin{equation}\label{eq:PGQ}
\begin{pmatrix}
& S_{q,k}^{(s)} &  M' &
\end{pmatrix}
=U
\begin{pmatrix}
& S_{q,k}^{(s)} &  M &
\end{pmatrix}
P. 
\end{equation}
Since $U$ is nonsingular,
it follows from the definition of $S_{q,k}$
that there is a $[k]_q \times [k]_q$ monomial matrix $P^*$ such that
\begin{equation}\label{eq:PGQ^*}
 S_{q,k}=(US_{q,k})P^*.
\end{equation}
Define the $n \times n$ monomial matrix $Q$ as follows:
\begin{equation}\label{eq:Q}
Q=
P^{-1}
\begin{pmatrix}
P^*    &&& \\
&\ddots&  &\\\
&      & P^* & \\
& && I_m
\end{pmatrix}.
\end{equation}
From~\eqref{eq:PGQ}, \eqref{eq:PGQ^*} and \eqref{eq:Q}, we have
\begin{align*}
\begin{pmatrix}
& S_{q,k}^{(s)} &  M' &
\end{pmatrix}
&=
\begin{pmatrix}
& (US_{q,k}P^*)^{(s)} & UMI_m &
\end{pmatrix}
Q^{-1} \\
&=
\begin{pmatrix}
& S_{q,k}^{(s)} & UM &
\end{pmatrix}
Q^{-1}.
\end{align*}
By Lemma~\ref{lem:1},
there is an $m \times m$ monomial matrix $R$ such that
$M'=UMR$.
This implies that $B \cong B'$.
\end{proof}

We are now in a position to give the following
characterization of LCD codes with large minimum weights,
which is the main result of this paper.

\begin{thm}\label{thm:main}
Suppose that $(q,k_0) \in \{(2,3),(3,2)\}$, $qd-(q-1)n \ge 1$,
$k \ge k_0$ and $q \cdot r_{q,n,k,d} \ge k$, where
$r_{q,n,k,d}$ is the integer defined in~\eqref{eq:r}.
\begin{enumerate}
\item
There is a one-to-one correspondence between
equivalence classes of LCD $[n,k,d]$ codes over $\FF_q$
with dual distances $d^\perp \ge 2$ 
and 
equivalence classes of LCD
$[q \cdot r_{q,n,k,d},k,(q-1)r_{q,n,k,d}]$ codes over $\FF_q$
with dual distances $d^\perp \ge 2$.
\item
If there is no LCD $[q \cdot r_{q,n,k,d},k,(q-1)r_{q,n,k,d}]$ code
over $\mathbb{F}_q$ with dual distance $d^\perp \ge 2$,
then there is no LCD $[n,k,d]$ code over $\mathbb{F}_q$.
\end{enumerate}
\end{thm}
\begin{proof}
The assertion (i) follows from Lemmas~\ref{lem:char}, \ref{lem:main1} and~\ref{lem:main2}.
The assertion (ii)  follows from Proposition~\ref{prop:dd1} and (i) 
(or Lemma~\ref{lem:char}).
\end{proof}

Roughly speaking,
Theorem~\ref{thm:main} (i) says that a classification of
LCD codes with large minimum weights for small length
yields that of LCD codes with large minimum weights for large length.
In addition, 
Theorem~\ref{thm:main} (ii) says that the nonexistence of
LCD codes with large minimum weights for small length
yields that of LCD codes with large minimum weights for large length.

\subsection{Modification of Theorem~\ref{thm:main}}
As the next step, we consider a modification of Theorem~\ref{thm:main}
by adding some assumption~\eqref{eq:as} on minimum weights (Theorem~\ref{thm:main2}).
Assume that we write 
\[
n=[k]_q \cdot s+t,
\]
where $s \in \ZZ_{\ge 0}$
and $t \in \{0,1,\ldots,[k]_q-1\}$.
In addition, assume the following:
\begin{equation}\label{eq:as}
\begin{split}
&\text{the minimum weight $d$ is written as}\\
&d(s,t)=q^{k-1}s+\alpha(t), \\
&\text{where $\alpha(t)$ is a constant depending on only $t$.} 
\end{split}
\end{equation}
The condition $q \cdot d(s,t)-(q-1)([k]_q \cdot s+t) \ge 1$ in 
Theorem~\ref{thm:main}
is equivalent to the condition
$s \ge s'_{q,([k]_q \cdot s+t),k,d(s,t)}$, where
\begin{equation}\label{eq:s0}
 s'_{q,([k]_q \cdot s+t),k,d(s,t)}=
  \frac{qr_{q,([k]_q \cdot s+t),k,d(s,t)}-t}{[k]_q}+1. 
\end{equation}
From~\eqref{eq:r}, we have 
\begin{equation}\label{eq:r2}
r_{q,([k]_q \cdot s+t),k,d(s,t)}
=q^{k-1}([k]_q \cdot s+t)-[k]_q \cdot d(s,t)
=q^{k-1}t -[k]_q \cdot \alpha(t).
\end{equation}
From~\eqref{eq:s0} and \eqref{eq:r2}, we have the following:

\begin{lem}\label{lem:rqcon}
Both $r_{q,([k]_q \cdot s+t),k,d(s,t)}$ 
and 
$s'_{q,([k]_q \cdot s+t),k,d(s,t)}$
depend on only $q,k,t$ and do not depend on $s$.
\end{lem}


\begin{thm}
\label{thm:main2}
Suppose that $(q,k_0) \in \{(2,3),(3,2)\}$ and $k \ge k_0$.
Write $n=[k]_q \cdot s+t \ge k$,
where $s \in \ZZ_{\ge 0}$
and $t \in \{0,1,\ldots,[k]_q-1\}$.
Assume~\eqref{eq:as}, that is,
$d$ is written as $d(s,t)=q^{k-1}s+\alpha(t)$,
where $\alpha(t)$ is a constant depending on only $t$.
Let $r$ and $s'$ denote the integers 
$r_{q,[k]_q \cdot s+t ,k,d(s,t)}$ and $s'_{q,[k]_q \cdot s+t ,k,d(s,t)}$ 
defined in~\eqref{eq:r} and~\eqref{eq:s0}, respectively.
Suppose that $q r \ge k$. 
\begin{enumerate}
\item
There is a one-to-one correspondence between
equivalence classes of LCD codes over $\FF_q$
with dual distances $d^\perp \ge 2$ and parameters
\begin{equation}\label{eq:mainp1}
[qr,k,(q-1)r]
=
[[k]_q \cdot (s'-1)+t,k,q^{k-1}(s'-1)+\alpha(t)],
\end{equation}
and 
equivalence classes of LCD codes over $\FF_q$
with dual distances $d^\perp \ge 2$ and parameters
\begin{equation}\label{eq:mainp2}
[[k]_q \cdot s+t,k,q^{k-1}s+\alpha(t)],
\end{equation}
for every integer
$s \ge s'$.
\item
If there is no LCD code over $\FF_q$
with dual distance $d^\perp \ge 2$ and parameters~\eqref{eq:mainp1},
then
there is no  LCD code over $\FF_q$ with parameters~\eqref{eq:mainp2}
for every integer $s \ge s'$.
\end{enumerate}
\end{thm}
\begin{proof}
From~\eqref{eq:s0}, we have
\begin{equation}\label{eq:q1}
q r
=[k]_q (s'-1)+t.
\end{equation}
From~\eqref{eq:r2} and \eqref{eq:q1}, we have
\begin{equation}\label{eq:q2}
\begin{split}
(q-1)r &= \frac{q-1}{q}([k]_q (s'-1)+t)\\ 
&=q^{k-1}(s'-1) +\frac{1}{q}(-(s'-1)+t(q-1))\\
&= q^{k-1} (s'-1)+\alpha(t).
\end{split}
\end{equation}
The assertion (i) follows from Theorem~\ref{thm:main},
Lemma~\ref{lem:rqcon}, \eqref{eq:q1} and \eqref{eq:q2}.

For $s \ge s'$, the assertion (ii)  follows directly from (i). 
Consider the case $s < s'-1$.
Suppose that
there is an LCD $[[k]_q \cdot s''+t,k,q^{k-1}s''+\alpha(t)]$ code over
$\mathbb{F}_q$ for some  $s'' < s'-1$.
Then, by Lemma~\ref{lem:190228-1},
there is an LCD $[[k]_q \cdot (s'-1)+t,k,q^{k-1}(s'-1)+\alpha(t)]$
code over $\mathbb{F}_q$.
This is a contradiction.
\end{proof}

The above theorem can be regarded as a slight improvement of
Theorem~\ref{thm:main} 
under the assumption~\eqref{eq:as}.
Roughly speaking,
Theorem~\ref{thm:main2} (i) says that a classification of
LCD codes with large minimum weights for small length
yields that of LCD codes with large minimum weights for
a family of lengths.
In addition, 
Theorem~\ref{thm:main2} (ii) says that the nonexistence of
LCD codes with large minimum weights for small length
yields that of LCD codes with large minimum weights for
a family of lengths.

By using Theorem~\ref{thm:main2},
for arbitrary $n$,
we complete the determination of 
the largest minimum weights $d_q(n,k)$
and  a classification of optimal 
LCD $[n,k]$ codes over $\FF_q$, where
$k \le 6-q$ and $q=2$ (resp.\ $q=3$)
in Section~\ref{sec:2} (resp.\ \ref{sec:3}).
We remark that the assumption~\eqref{eq:as} is automatically satisfied
for our study in these sections.


\section{Classification method}
\label{sec:Cmethod}

In Sections~\ref{sec:2} and~\ref{sec:3},
by Theorem~\ref{thm:main},
for arbitrary $n$, 
we complete the determination of 
the largest minimum weights $d_q(n,k)$
and  a classification of optimal 
LCD $[n,k]$ codes over $\FF_q$, where
$k \le 6-q$ and $q \in \{2,3\}$.
These results are obtained by showing 
the nonexistence of
LCD codes with large minimum weights
and giving a classification of optimal
LCD codes for small lengths by the following method.

Here we suppose that $(q,k) \in \{(2,3),(2,4),(3,2),(3,3)\}$.
Any LCD $[n,k]$ code over $\FF_q$ with dual distance $d^\perp \ge 2$
is equivalent to a code $C_{q,k}(m)$ which has generator matrix
$G_{q,k}(m)$ of form~\eqref{eq:Gqkm}
for some $m=(m_1,m_2,\ldots,m_{[k]_q}) \in \ZZ_{\ge 0}^{[k]_q}$ and
$n=\sum_{i=1}^{[k]_q}m_i$.
In addition,
any $[n,k]$ code over $\FF_q$ is equivalent to some
code with generator matrix of form 
$
\left(\begin{array}{cc}
I_k &  A 
\end{array}\right),
$
where $A$ is a $k \times (n-k)$ matrix.
Hence, we may assume without loss of generality that
\[
\begin{array}{rr}
m_1,m_2,m_4 \ge 1 \text{ if } (q,k)=(2,3), &
m_1,m_2,m_4,m_8 \ge 1 \text{ if } (q,k)=(2,4), \\
m_1,m_2 \ge 1 \text{ if } (q,k)=(3,2), &
m_1,m_2,m_5 \ge 1 \text{ if } (q,k)=(3,3). 
\end{array}
\]
Fix parameters $q,n,k$ and $d$.
By considering all possible vectors $m$ satisfying the above condition
and~\eqref{eq:mi},
we found all LCD $[n,k,d]$ codes over $\FF_q$
which must be checked further for equivalences.
For calculations of determinants of $GG^T$ for generator matrices $G$,
the NTL function {\tt determinant}~\cite{ntl} was used. 
Of course, if we found no LCD $[n,k,d']$ code over $\FF_q$
for $d \le d' \le g_q(n,k)$, then
we have $d_q(n,k) \le d-1$.
After checking whether codes are equivalent or not, 
we complete a classification of LCD $[n,k,d]$ codes over $\FF_q$.
We consider the corresponding digraphs for equivalence testing,
and we used {\tt nauty}~\cite{nauty} for digraph isomorphism testing
(see~\cite[Section~5]{AH-C} for the details).

All computer calculations in this paper were
done by programs in the language C.
Some verification was done by {\sc Magma}~\cite{Magma}.
Especially, 
when a classification
of LCD $[n,k,d]$ codes over $\FF_q$ with dual distances $d^\perp \ge 2$
was done, we verified by {\sc Magma} that 
all our codes are
LCD $[n,k,d]$ codes over $\FF_q$ with dual distances $d^\perp \ge 2$
and inequivalent codes.

\section{Binary optimal LCD codes}
\label{sec:2}

\subsection{Binary optimal LCD codes of dimension 3}

The largest minimum weights $d_2(n,3)$ were determined in~\cite{HS}
for arbitrary $n$ (see Table~\ref{Tab:2-3} for $d_2(n,3)$).
In addition, 
a classification of binary optimal LCD $[n,3]$ codes
was given in~\cite{HS} for $n \equiv 0,2,3,5 \pmod{7}$.
In this subsection, 
a classification of binary optimal LCD $[n,3]$ codes
is given for $n \equiv 1,4,6 \pmod{7}$.

\begin{table}[thbp]
\caption{Binary optimal LCD $[n,3]$ codes with dual distances $d^\perp\ge 2$}
\label{Tab:2-3}
\begin{center}
{\small
\begin{tabular}{c|c|c|c|c|c}
\noalign{\hrule height0.8pt}
 $n$ & $d_2(n,3)$ & $s'$ & $r$ & $N_2(n,3)$  & Reference\\
\hline
 $7s$   &$ 4s-1 $&$3$	 &   7& 1  &\cite{HS}\\
 $7s+1$ &$ 4s-1 $&$4$	 &  11& 7  & Proposition~\ref{prop:2-3-2}\\
 $7s+2$ &$ 4s   $&$3$	 &   8& 1  &\cite{HS}\\
 $7s+3$ &$ 4s+1 $&$2$	 &   5& 1  &\cite{HS}\\
 $7s+4$ &$ 4s+1 $&$3$	 &   9& 5  & Proposition~\ref{prop:2-3-2}\\
 $7s+5$ &$ 4s+2 $&$2$	 &   6& 1  &\cite{HS}\\
 $7s+6$ &$ 4s+2 $&$3$	 &  10& 5  & Proposition~\ref{prop:2-3-2}\\
\noalign{\hrule height0.8pt}
\end{tabular}
}
\end{center}
\end{table}

We apply Theorem~\ref{thm:main2} (i) to the case $(q,k)=(2,3)$.
For $n \ge 4$,
write $n=7s+t$, where $s \in \ZZ_{\ge 0}$ and $t \in \{1,4,6\}$.
Let $r=r_{2,7s+t,3,d_2(7s+t,3)}$ and $s'=s'_{2,7s+t,3,d_2(7s+t,3)}$
be the integers defined
in~\eqref{eq:r} and~\eqref{eq:s0}, respectively.
For each $7s+t$, we list $d_2(7s+t,3)$, $s'$ and $r$ in Table~\ref{Tab:2-3}.
Then $d_2(7s+t,3)$ is written as 
$4s+\alpha(t)$, where $\alpha(t)$ is a constant depending on only $t$.
Since the minimum weight satisfies 
the assumption~\eqref{eq:as} in Theorem~\ref{thm:main2},
we have the following:

\begin{prop}\label{prop:2-3}
There is a one-to-one correspondence between
equivalence classes of binary LCD
$[2r,3,r]$ codes with dual distances $d^\perp \ge 2$
and 
equivalence classes of binary LCD $[7s+t,3,d_2(7s+t,3)]$ codes 
with dual distances $d^\perp \ge 2$ 
for every integer $s \ge s'$. 
\end{prop}


\begin{table}[thbp]
\caption{$N_2(n,3)$}
\label{Tab:2-3-2}
\begin{center}
{\small
\begin{tabular}{c|c||c|c||c|c||c|c}
\noalign{\hrule height0.8pt}
 $n$ & $N_2(n,3)$ & $n$ & $N_2(n,3)$ &
 $n$ & $N_2(n,3)$ & $n$ & $N_2(n,3)$ \\
\hline
 4&  1& 10&  1& 16&  1& 22&  7 \\
 5&  1& 11&  5& 17&  1& 23&  1 \\
 6&  2& 12&  1& 18&  5& 24&  1 \\
 7&  1& 13&  5& 19&  1& 25&  5 \\
 8&  2& 14&  1& 20&  5& &\\
 9&  1& 15&  7& 21&  1& &\\
\noalign{\hrule height0.8pt}
 \end{tabular}
}
\end{center}
\end{table}

A classification of binary optimal LCD $[n,3]$ codes 
was done for $n \le 25$~\cite[Table~3]{AH-C}. 
We calculated the numbers $N_2(n,3)$ of inequivalent
binary optimal LCD $[n,3]$ codes with dual distances
$d^\perp \ge 2$, where the numbers $N_2(n,3)$ are listed in
Table~\ref{Tab:2-3-2}.

\begin{itemize}
\item Case $s \ge s'$:
Since $N_2(18,3)=5$, $N_2(20,3)=5$ and $N_2(22,3)=7$,
by Proposition~\ref{prop:2-3},
we have a classification of binary optimal
LCD $[7s+t,3]$ codes with dual distances $d^\perp \ge 2$ for
$s \ge s'$.

\item Case $s < s'$:
From Table~\ref{Tab:2-3-2},
it is known that
$N_{2}(15,3)=N_{2}(22,3)=7$,
$N_{2}(11,3)=N_{2}(18,3)=5$ and 
$N_{2}(13,3)=N_{2}(20,3)=5$.
\end{itemize}
Therefore, we have the following:

\begin{prop}\label{prop:2-3-2}
\begin{enumerate}
\item There are $7$ inequivalent binary optimal LCD $[7s+1,3]$
codes with dual distances $d^\perp\ge 2$ for every integer $s \ge 2$.
\item Suppose that $t \in \{4,6\}$.
Then there are $5$ inequivalent binary optimal LCD $[7s+t,3]$
codes with dual distances $d^\perp\ge 2$ for every integer $s \ge 1$.
\end{enumerate}
\end{prop}

\begin{rem}
If there is a binary $[n,n-3,d]$ code with $d \ge 3$, then
$n \le 7$ by the sphere-packing bound. 
Thus, the dual distances of all the codes in the above proposition
are exactly $2$.
\end{rem}

By Lemma~\ref{lem:class}, the above proposition
completes a classification of binary
optimal LCD codes of dimension $3$.

\subsection{Binary optimal LCD codes of dimension 4}

In this subsection, we give
a classification of binary optimal LCD codes of dimension $4$.

\subsubsection{Determination of $d_2(n,4)$}
\label{sec:2-4-D}

It is known that
\begin{align*}
& d_2(n,4)=
\begin{cases}
\left\lfloor \frac{8n}{15} \right\rfloor
 &\text{ if }n \equiv 5,9,13 \pmod{15},\\
\left\lfloor \frac{8n}{15} \right\rfloor-1
 &\text{ if }n \equiv 2,3,4,6,10 \pmod{15},
\end{cases}
\\& d_2(n,4)\ge
\begin{cases}
\left\lfloor \frac{8n}{15} \right\rfloor-1
&\text{ if }n \equiv 1,7,8,11,12,14 \pmod{15}, \\
\left\lfloor \frac{8n}{15} \right\rfloor-2
&\text{ if }n \equiv 0 \pmod{15},
\end{cases}
\end{align*}
\cite[Theorem~1 and Proposition~1]{AH}.


By Lemma~\ref{lem:n0},
there is no binary LCD $[15s,4,8s]$ code
for every positive integer $s$.
We apply Theorem~\ref{thm:main2} (ii) to the case $(q,k)=(2,4)$.
For $n \ge 7$,
write $n=15s+t$, where $s \in \ZZ_{\ge 0}$ and $t \in \{0,1,7,8,11,12,14\}$.
Suppose that $d(s,t)=8s-1$ if $t=0$ and $d(s,t)=g_2(15s+t,4)$
otherwise.
Let $r=r_{2,15s+t,4,d(s,t)}$ 
be the integer defined
in~\eqref{eq:r}.
For each $15s+t$, we list $d(s,t)$ and $r$
in Table~\ref{Tab:2-4-0}.
Then $d(s,t)$ is written as
$8s+\alpha(t)$, where $\alpha(t)$ is a constant depending on only $t$.
Since the minimum weight satisfies 
the assumption~\eqref{eq:as} in Theorem~\ref{thm:main2},
we have the following:

\begin{prop}\label{prop:2-4-0}
If there is no binary LCD $[2r,4,r]$ code
with dual distance $d^\perp \ge 2$,
then there is no binary LCD $[15s+t,4,d(s,t)]$ code
for every nonnegative integer $s$.
\end{prop}

\begin{table}[thbp]
\caption{$d(s,t)$ and $r$ in Proposition~\ref{prop:2-4-0}}
\label{Tab:2-4-0}
\begin{center}
{\small
\begin{tabular}{c|c|c||c|c|c}
\noalign{\hrule height0.8pt}
 $n$ & $d(s,t)$ & $r$ & $n$ & $d(s,t)$ & $r$ \\
\hline
 $15s$   & $8s-1$ &  15 & $15s+11$& $8s+5$ &  13 \\
 $15s+1$ & $8s  $ &   8 & $15s+12$& $8s+6$ &   6 \\
 $15s+7$ & $8s+3$ &  11 & $15s+14$& $8s+7$ &   7 \\
 $15s+8$ & $8s+4$ &   4 & & &     \\
 \noalign{\hrule height0.8pt}
\end{tabular}
}
\end{center}
\end{table}

It is known that
there is no binary LCD $[2r,4,r]$ code for
\[
r \in \{4,6,7,8,11,13,15\},
\]
\cite[Table~14 and Proposition~2]{AH} and~\cite[Table~3]{HS}.
By Proposition~\ref{prop:2-4-0}, 
we have the following:

\begin{prop}
Suppose that $n \ge 4$.  Then
\[
d_2(n,4)=
\begin{cases}
\left\lfloor \frac{8n}{15} \right\rfloor & \text{if } n \equiv 5,9,13 \pmod{15}, \\
\left\lfloor \frac{8n}{15} \right\rfloor-1 & \text{if } n \equiv 1,2,3,4,6,7,8,10,11,12,14 \pmod{15}, \\
\left\lfloor \frac{8n}{15} \right\rfloor-2 & \text{if } n \equiv 0 \pmod{15}.
\end{cases}
\]
\end{prop}

\subsubsection{Classification of binary optimal LCD codes of dimension 4}

We apply Theorem~\ref{thm:main2} (i) to the case $(q,k)=(2,4)$.
For $n \ge 4$,
write $n=15s+t$, where $s \in \ZZ_{\ge 0}$ and $t \in \{0,1,\ldots,14\}$.
Let $r=r_{2,15s+t,4,d_2(15s+t,4)}$ and $s'=s'_{2,15s+t,4,d_2(15s+t,4)}$
be the integers defined
in~\eqref{eq:r} and~\eqref{eq:s0}, respectively.
For each $15s+t$, we list $d_2(15s+t,4)$, $s'$ and $r$ in
Table~\ref{Tab:2-4}.
Then $d_2(15s+t,4)$ is written as
$8s+\alpha(t)$, where $\alpha(t)$ is a constant depending on only $t$.
Since the minimum weight satisfies 
the assumption~\eqref{eq:as} in Theorem~\ref{thm:main2},
we have the following:

\begin{prop}\label{prop:2-4}
There is a one-to-one correspondence between
equivalence classes of binary LCD
$[2r,4,r]$ codes with dual distances $d^\perp \ge 2$
and
equivalence classes of binary LCD $[15s+t,4,d_2(15s+t,4)]$ codes 
with dual distances $d^\perp \ge 2$
for every integer $s \ge s'$.
\end{prop}

\begin{table}[thbp]
\caption{Binary optimal LCD $[n,4]$ codes with dual distances $d^\perp\ge 2$}
\label{Tab:2-4}
\begin{center}
{\small
\begin{tabular}{c|c|c|c|c}
\noalign{\hrule height0.8pt}
 $n$ & $d_2(n,4)$&$s'$  & $r$ &  $N_2(n,4)$\\
\hline
 $15s$   & $ 8s-2 $&$ 5$  & 30 & 404 \\
 $15s+1$ & $ 8s-1 $&$ 4$  & 23 &  10 \\
 $15s+2$ & $ 8s   $&$ 3$  & 16 &   2 \\
 $15s+3$ & $ 8s   $&$ 4$  & 24 &  39 \\
 $15s+4$ & $ 8s+1 $&$ 3$  & 17 &   2 \\
 $15s+5$ & $ 8s+2 $&$ 2$  & 10 &   1 \\
 $15s+6$ & $ 8s+2 $&$ 3$  & 18 &  10 \\
 $15s+7$ & $ 8s+2 $&$ 4$  & 26 & 121 \\
 $15s+8$ & $ 8s+3 $&$ 3$  & 19 &   2 \\
 $15s+9$ & $ 8s+4 $&$ 2$  & 12 &   1 \\
 $15s+10$& $ 8s+4 $&$ 3$  & 20 &  11 \\
 $15s+11$& $ 8s+4 $&$ 4$  & 28 & 151 \\
 $15s+12$& $ 8s+5 $&$ 3$  & 21 &   9 \\
 $15s+13$& $ 8s+6 $&$ 2$  & 14 &   2 \\
 $15s+14$& $ 8s+6 $&$ 3$  & 22 &  33 \\
 \noalign{\hrule height0.8pt}
\end{tabular}
}
\end{center}
\end{table}

\begin{table}[thbp]
\caption{$N_2(n,4)$}
\label{Tab:2-4-2}
\begin{center}
{\small
 \begin{tabular}{c|c||c|c||c|c||c|c}
\noalign{\hrule height0.8pt}
$n$ & $N_2(n,4)$ &$n$ & $N_2(n,4)$ &
$n$ & $N_2(n,4)$ &$n$ & $N_2(n,4)$ \\
\hline
 5&  1 &16&  7 &27&  9 &40& 11 \\
 6&  3 &17&  2 &28&  2 &41&151 \\
 7&  5 &18& 20 &29& 33 &42&  9 \\
 8&  1 &19&  2 &30&310 &44& 33 \\
 9&  1 &20&  1 &31& 10 &45&404 \\
10&  4 &21& 10 &32&  2 &46& 10 \\
11& 15 &22& 76 &33& 39 &48& 39 \\
12&  6 &23&  2 &34&  2 &52&121 \\
13&  2 &24&  1 &36& 10 &56&151 \\
14& 14 &25& 11 &37&121 &60&404 \\
15& 73 &26&106 &38&  2 &&\\
\noalign{\hrule height0.8pt}
 \end{tabular}
}
\end{center}
\end{table}

A classification of binary LCD codes was done in~\cite[Table~6]{AH-C}
for lengths up to $13$.
We calculated the numbers $N_2(n,4)$ of the inequivalent
binary optimal LCD $[n,4]$ codes with dual distances
$d^\perp \ge 2$, where
the numbers $N_2(n,4)$ are listed in Table~\ref{Tab:2-4-2}.

\begin{itemize}
\item Case $s \ge s'$:
By the method given in Section~\ref{sec:Cmethod},
our computer search
completes a classification of binary
LCD $[2r,4,r]$ codes with dual distances
$d^\perp \ge 2$ for 
\begin{equation*}\label{eq:2-4-r}
 r\in \{10,
 12,
 14,
 16,
 17,
 18,
 19,
 20,
 21,
 22,
 23,
 24,
 26,
 28,
 30\}.
\end{equation*}
The numbers $N_2(2r,4)$ of the inequivalent
binary LCD $[2r,4,r]$ codes with dual distances
$d^\perp \ge 2$ are listed in Table~\ref{Tab:2-4-2}.
These codes $B_{2r,4,i}$ $(i\in\{1,2,\ldots,N_2(2r,4)\})$
are presented by codes $C_{2,4}(m)$
with generator matrices $G_{2,4}(m)$ of form~\eqref{eq:Gqkm}.
For $r \in \{10,12,14,16,17,19\}$,
in order to display the codes $B_{2r,4,i}$,
the vectors $m$
are listed in Table~\ref{Tab:2-4-3}.
For all the parameters $r$, the vectors $m$
are available at
\url{http://www.math.is.tohoku.ac.jp/~mharada/AHS}.

\begin{table}[thb]
\caption{$B_{n,4,i}$}
\label{Tab:2-4-3}
\begin{center}
{\small
\begin{tabular}{c|c|c}
\noalign{\hrule height0.8pt}
$C_{2,4}(m)$ & $(2r,r)$ & $m$ \\\hline
$B_{20,4,1}$&$(20,10)$ &$(2,2,1,2,1,1,1,2,1,1,1,1,1,1,2)$\\
$B_{24,4,1}$&$(24,12)$ &$(2,2,2,2,2,1,1,2,1,2,1,2,1,1,2)$\\
$B_{28,4,1}$&$(28,14)$ &$(3,3,1,3,1,1,2,3,1,1,2,1,2,2,2)$\\
$B_{28,4,2}$&          &$(2,2,1,3,2,2,2,3,2,2,2,1,1,1,2)$\\
$B_{32,4,1}$&$(32,16)$ &$(3,3,2,3,2,1,2,3,1,2,2,2,2,2,2)$\\
$B_{32,4,2}$&          &$(3,3,1,3,2,1,3,3,1,2,3,2,2,2,1)$\\
$B_{34,4,1}$&$(34,17)$ &$(3,3,2,3,2,2,2,3,2,2,2,2,2,2,2)$\\
$B_{34,4,2}$&          &$(3,3,2,3,2,2,1,3,2,2,2,2,2,2,3)$\\
$B_{38,4,1}$&$(38,19)$ &$(3,3,3,3,3,2,2,3,2,3,2,3,2,2,2)$\\
$B_{38,4,2}$&          &$(3,3,3,3,3,2,1,3,2,3,2,3,2,2,3)$\\
 \hline
$C_{2,4}(m)$ & $(n,d)$ & $m$ \\\hline
$B_{17,4,1}$&$(17,8)$ & $(2,2,1,2,1,0,1,2,0,1,1,1,1,1,1)$\\
$B_{17,4,2}$&         & $(2,2,0,2,1,0,2,2,0,1,2,1,1,1,0)$\\
$B_{19,4,1}$&$(19,9)$ & $(2,2,1,2,1,1,1,2,1,1,1,1,1,1,1)$\\
$B_{19,4,2}$&         & $(2,2,1,2,1,1,0,2,1,1,1,1,1,1,2)$\\
$B_{23,4,1}$&$(23,11)$& $(2,2,2,2,2,1,1,2,1,2,1,2,1,1,1)$\\
$B_{23,4,2}$&         & $(2,2,2,2,2,1,0,2,1,2,1,2,1,1,2)$\\
 \noalign{\hrule height0.8pt}
 \end{tabular}
}
\end{center}
\end{table}

      By Proposition~\ref{prop:2-4},
we have  a classification of binary optimal
LCD $[15s+t,4]$ codes with dual distances $d^\perp \ge 2$
for $s \ge s'$.

\item Case $s < s'$:
By the method given in Section~\ref{sec:Cmethod},
our computer search
completes a classification of binary optimal
LCD $[n,4]$ codes with dual distances
$d^\perp \ge 2$ for 
\begin{equation*}\label{eq:2-4-n}
n \in \{14,15,16,17,18,19,21,22,23,25,
26,27,29,30,31,33,37,41,45\}.
\end{equation*}
The numbers $N_2(n,4)$ are listed in Table~\ref{Tab:2-4-2}.
These codes $B_{n,4,i}$ $(i\in\{1,2,\ldots,N_2(n,4)\})$
are presented by codes $C_{2,4}(m)$
with generator matrices $G_{2,4}(m)$ of form~\eqref{eq:Gqkm}.
For $(n,d) \in \{(17,8),(19,9),(23,11)\}$,
in order to display the codes $B_{n,4,i}$,
the vectors $m$
are listed in Table~\ref{Tab:2-4-3}.
For all the lengths $n$, the vectors $m$
are available at
\url{http://www.math.is.tohoku.ac.jp/~mharada/AHS}.
\end{itemize}
Therefore, we have the following:

\begin{prop}\label{prop:2-4-2}
\begin{enumerate}
\item
There are $404$ inequivalent binary optimal LCD $[15s,4]$
codes with dual distances $d^\perp\ge 2$ for every integer $s \ge 3$.
\item Suppose that $t \in \{1,6\}$.
Then there are $10$ inequivalent binary optimal LCD $[15s+t,4]$
codes with dual distances $d^\perp\ge 2$ for every integer $s \ge 2$
if $t=1$ and  $s \ge 1$ if $t=6$.

\item Suppose that $t \in \{2,4,8,13\}$.
Then
there are $2$ inequivalent binary optimal LCD $[15s+t,4]$
codes with dual distances $d^\perp\ge 2$ for every integer
$s \ge 1$ if $t\in\{2,4,8\}$ and $s \ge 0$ if $t=13$.
\item
There are $39$ inequivalent binary optimal LCD $[15s+3,4]$
codes with dual distances $d^\perp\ge 2$ for every integer $s \ge 2$.
\item Suppose that $t \in \{5,9\}$.
Then
there is a unique  binary optimal LCD $[15s+t,4]$
code with dual distance $d^\perp\ge 2$, up to equivalence,
for every integer $s \ge 2$.

\item
There are $121$ inequivalent binary optimal LCD $[15s+7,4]$
codes with dual distances $d^\perp\ge 2$ for every integer $s \ge 2$.
\item
There are $11$ inequivalent binary optimal LCD $[15s+10,4]$
codes with dual distances $d^\perp\ge 2$ for every integer $s \ge 1$.
\item
There are $151$ inequivalent binary optimal LCD $[15s+11,4]$
codes with dual distances $d^\perp\ge 2$ for every integer $s \ge 2$.
\item
There are $9$ inequivalent binary optimal LCD $[15s+12,4]$
codes with dual distances $d^\perp\ge 2$ for every integer $s \ge 1$.
\item
There are $33$ inequivalent binary optimal LCD $[15s+14,4]$
codes with dual distances $d^\perp\ge 2$ for every integer $s \ge 1$.
 \end{enumerate}
\end{prop}

\begin{rem}
If there is a binary $[n,n-4,d]$ code with $d \ge 3$, then
$n \le 15$ by the sphere-packing bound. 
Thus, the dual distances of all the codes in the above proposition
are exactly $2$.
\end{rem}

By Lemma~\ref{lem:class}, the above proposition
completes a classification of binary
optimal LCD codes of dimension $4$.

\section{Ternary optimal LCD codes}
\label{sec:3}

\subsection{Ternary optimal LCD codes of dimension 2}

In this subsection, we give
a classification of ternary optimal LCD codes of dimension $2$.

\subsubsection{Determination of $d_3(n,2)$}
Here we show that
$d_3(n,2)=\left\lfloor \frac{3n}{4} \right\rfloor$ if $n \equiv 1,2 \pmod{4}$ and
$d_3(n,2)=\left\lfloor \frac{3n}{4} \right\rfloor-1$ if $n \equiv 0,3 \pmod{4}$.

It is trivial that $\FF_3^2$ is the ternary LCD $[2,2,1]$ code.
By~\cite[Proposition~5 and Table~4]{AH-C},
there is a ternary LCD $[n,2,d]$ code for
$(n,d)\in\{(3,1),(4,2),(5,3)\}$.
Suppose that $d=g_3(n,2)$ if $n \equiv 1,2 \pmod{4}$ and 
$d=g_3(n,2)-1$ otherwise, 
where the values $g_3(n,2)$ are listed in  Table~\ref{table:G-bound-3-2}.
By Lemma~\ref{lem:190228-1}, ternary LCD $[n,2,d]$ codes
are constructed for $n \ge 2$.

\begin{table}[thb]
\caption{$g_3(n,2)$}
\label{table:G-bound-3-2}
\begin{center}
{\small
\begin{tabular}{c|c||c|c}
\noalign{\hrule height0.8pt}
$n$ & $g_3(n,2)$ & $n$ & $g_3(n,2)$ \\
\hline
$4s$   & $3s$   & $4s+2$ & $3s+1$ \\ 
$4s+1$ & $3s$   & $4s+3$ & $3s+2$ \\
\noalign{\hrule height0.8pt}
\end{tabular}
}
\end{center}
\end{table}

By Lemma~\ref{lem:n0},
there is no ternary LCD $[4s,2,3s]$ code
for every positive integer $s$.
We investigate the existence of a ternary LCD $[4s+3,2,g_3(4s+3,2)]$ code.
Then $g_3(4s+3,2)$ is written as $3s+2$.
Thus, the minimum weight satisfies 
the assumption~\eqref{eq:as} in Theorem~\ref{thm:main2}.
Also, there is no ternary LCD $[3,2,2]$ code~\cite[Proposition~5]{AH-C}.
Hence, by applying Theorem~\ref{thm:main2} (ii), we have the following:

\begin{prop}
There is no ternary LCD $[4s+3,2,3s+2]$ code
for every positive integer $s$.
\end{prop}


Therefore, we have the following:
\begin{prop}
Suppose that $n \ge 2$.  Then
\[
d_3(n,2)=
\begin{cases}
\left\lfloor \frac{3n}{4} \right\rfloor & \text{if } n \equiv 1,2 \pmod{4}, \\
\left\lfloor \frac{3n}{4} \right\rfloor-1 & \text{if } n \equiv 0,3 \pmod{4}.
\end{cases}
\]
\end{prop}

\begin{rem}
 After we completed this work, we became aware of~\cite{PZK},
 where the above proposition was given by
 using an elementary method.
\end{rem}

\subsubsection{Classification of ternary optimal LCD codes of dimension 2}

We apply Theorem~\ref{thm:main2} (i) to the case $(q,k)=(3,2)$.
For $n \ge 2$,
write $n=4s+t$, where $s \in \ZZ_{\ge 0}$ and $t \in \{0,1,2,3\}$.
Let $r=r_{3,4s+t,2,d_3(4s+t,2)}$ and $s'=s'_{3,4s+t,2,d_3(4s+t,2)}$
be the integers defined
in~\eqref{eq:r} and~\eqref{eq:s0}, respectively.
For each $4s+t$, we list $d_3(4s+t,2)$, $s'$ and $r$ in Table~\ref{Tab:3-2}.
Then $d_3(4s+t,2)$ is written as
$3s+\alpha(t)$, where $\alpha(t)$ is a constant depending on only $t$.
Since the minimum weight satisfies 
the assumption~\eqref{eq:as} in Theorem~\ref{thm:main2},
we have the following:

\begin{prop}\label{prop:3-2}
There is a one-to-one correspondence between
equivalence classes of ternary LCD
 $[3r,2,2r]$ codes with dual distances $d^\perp \ge 2$
 and
 equivalence classes of ternary LCD $[4s+t,2,d_3(4s+t,2)]$ codes
with dual distances $d^\perp \ge 2$
for every integer $s \ge s'$.
\end{prop}

\begin{table}[thbp]
\caption{Ternary optimal LCD $[n,2]$ codes with  dual distances $d^\perp\ge 2$}
\label{Tab:3-2}
\begin{center}
{\small
\begin{tabular}{c|c|c|c|c}
\noalign{\hrule height0.8pt}
 $n$ &$d_3(n,2)$& $s'$ & $r$ & $N_3(n,2)$\\
\hline
 $4s  $ & $3s-1$ &$4$ & 4 & 2\\
 $4s+1$ & $3s$   &$3$ & 3 & 1\\
 $4s+2$ & $3s+1$ &$2$ & 2 & 1\\
 $4s+3$ & $3s+1$ &$4$ & 5 & 3\\
\noalign{\hrule height0.8pt}
\end{tabular}
}
\end{center}
\end{table}

\begin{table}[thbp]
\caption{$N_3(n,2)$}
\label{Tab:3-2-2}
\begin{center}
{\small
\begin{tabular}{c|c||c|c||c|c||c|c}
\noalign{\hrule height0.8pt}
$n$ & $N_3(n,2)$&$n$ & $N_3(n,2)$&
$n$ & $N_3(n,2)$&$n$ & $N_3(n,2)$\\
\hline
  3& 1&  6& 1&  9& 1& 12& 2\\
  4& 2&  7& 2& 10& 1& 15& 3\\
  5& 1&  8& 2& 11& 3&   &  \\
\noalign{\hrule height0.8pt}
\end{tabular}
}
\end{center}
\end{table}

A classification of ternary LCD codes
is known~\cite[Table~4]{AH-C} for lengths up to $10$.
We calculated the numbers $N_3(3r,2)$ of inequivalent
ternary LCD $[3r,2,2r]$ codes with dual distances $d^\perp \ge 2$,
where the numbers $N_3(3r,2)$ are listed in Table~\ref{Tab:3-2-2}.

\begin{itemize}
\item Case $s \ge s'$:
From Table~\ref{Tab:3-2-2},
it is known that $N_3(6,2)=1$ and $N_3(9,2)=1$.
By the method given in Section~\ref{sec:Cmethod},
our computer search completes a classification of ternary 
LCD $[3r,2,2r]$ codes with dual distances
$d^\perp \ge 2$ for $r \in \{4,5\}$. 
The numbers $N_3(3r,2)$ are listed in Table~\ref{Tab:3-2-2}.
For $n \in \{12,15\}$,
the codes $T_{n,2,i}$ $(i\in\{1,2,\ldots,N_3(n,2)\})$
are presented by codes $C_{3,2}(m)$
with generator matrices $G_{3,2}(m)$ of form~\eqref{eq:Gqkm}.
In order to display the codes $T_{n,2,i}$,
the vectors $m$
are listed in Table~\ref{Tab:3-2-3}.
By Proposition~\ref{prop:3-2},
we have a classification of ternary optimal
LCD $[4s+t,2]$ codes  with dual distances $d^\perp \ge 2$
for $s \ge s'$.

\item Case $s < s'$:
By the method given in Section~\ref{sec:Cmethod},
our computer search shows that there are $3$ inequivalent
ternary optimal LCD $[11,2,7]$ codes with dual distances
$d^\perp \ge 2$.
The codes $T_{11,2,i}$ $(i\in\{1,2,3\})$
are presented by codes $C_{3,2}(m)$
with generator matrices $G_{3,2}(m)$ of form~\eqref{eq:Gqkm}.
In order to display the codes $T_{11,2,i}$,
the vectors $m$ are listed in Table~\ref{Tab:3-2-3}.
\end{itemize}
Therefore, we have the following:

%

\begin{prop}\label{prop:3-2-2}
\begin{enumerate}
\item
There are $2$ inequivalent ternary optimal LCD
$[4s,2]$ codes
with dual distances $d^\perp\ge 2$ for every integer $s \ge 1$.

\item Suppose that $t \in \{1,2\}$.  Then
there is a unique ternary optimal LCD
$[4s+t,2]$ code with dual distance $d^\perp\ge 2$,
up to equivalence, for every integer $s \ge 1$.

\item There are $3$ inequivalent ternary optimal LCD
$[4s+3,2]$ codes
with dual distances $d^\perp\ge 2$ for every integer $s \ge 2$.
 \end{enumerate}
\end{prop}

\begin{rem}
If there is a ternary $[n,n-2,d]$ code with $d \ge 3$, then
$n \le 4$ by the sphere-packing bound. 
Since $d_3(4,2) \le 2$,
the dual distances of all the codes in the above proposition
are exactly $2$.
\end{rem}

By Lemma~\ref{lem:class}, the above proposition
completes a classification of ternary optimal LCD codes of
dimension $2$.

\begin{table}[thb]
\caption{$T_{n,2,i}$}
\label{Tab:3-2-3}
\begin{center}
{\small
\begin{tabular}{c|c|c||c|c|c}
\noalign{\hrule height0.8pt}
$C_{3,2}(m)$ & $(n,d)$ & $m$ &
$C_{3,2}(m)$ & $(n,d)$ & $m$ \\\hline
$T_{11,2,1}$ &$(11,7)$& $(4,4,3,0)$&$T_{12,2,2}$ &$(12,8)$ & $(3,4,3,2)$\\
$T_{11,2,2}$ &$(11,7)$& $(3,4,3,1)$&$T_{15,2,1}$ &$(15,10)$& $(5,5,4,1)$\\
$T_{11,2,3}$ &$(11,7)$& $(3,4,2,2)$&$T_{15,2,2}$ &$(15,10)$& $(4,5,4,2)$\\
$T_{12,2,1}$ &$(12,8)$& $(4,4,2,2)$&$T_{15,2,3}$ &$(15,10)$& $(4,5,3,3)$\\
 \noalign{\hrule height0.8pt}
 \end{tabular}
}
\end{center}
\end{table}

\subsection{Ternary optimal LCD codes of dimension 3}

In this subsection, we give
a classification of ternary optimal LCD codes of dimension $3$.

\subsubsection{Determination of $d_3(n,3)$}

Here we show that
$d_3(n,3)=
\left\lfloor \frac{9n}{13} \right\rfloor$  if  $n \equiv 4,7,10 
\pmod{13}$ and 
$d_3(n,3)=
\left\lfloor \frac{9n}{13} \right\rfloor-1$  if $n \equiv 
0,1,2,3,5,6,8,9,11,12 \pmod{13}$.

It is trivial that $\FF_3^3$ is the ternary LCD $[3,3,1]$ code.
By~\cite[Proposition~5 and
Table~4]{AH-C}, there is a ternary LCD $[n,3,d]$ code for
$(n,d)\in\{(4,2),(5,2),(6,3),(7,4),(8,4),(9,5),(10,6)\}$.
By considering codes $C_{3,3}(m)$
with generator matrices $G_{3,3}(m)$ of form~\eqref{eq:Gqkm},
we found ternary LCD $[n,3,d]$ codes $T_{n,3}$ for
$(n,d)\in\{(11,6),(12,7),(13,8),(14,8),(15,9)\}$.
In order to display the codes $T_{n,3}$,
the vectors $m$ are
listed in Table~\ref{table:190228-1}.
Suppose that $d=g_3(n,3)$ if  $n \equiv 2,3,4,6,7,10 \pmod{13}$ and
$d=g_3(n,3)-1$ otherwise,
where the values $g_3(n,3)$ are listed in Table~\ref{table:G-bound-3-3}.
By Lemma~\ref{lem:190228-1}, ternary LCD $[n,3,d]$ codes are constructed
for $n \ge 3$.

\begin{table}[thb]
\caption{$T_{n,3}$}
\label{table:190228-1}
\begin{center}
{\small
\begin{tabular}{c|c|c}
\noalign{\hrule height0.8pt}
$C_{3,3}(m)$ & $(n,d)$ & $m$ \\\hline
$T_{11,3}$ & $(11,6)$ & $(1, 2, 2, 0, 1, 1, 0, 0, 2, 0, 0, 0, 2)$\\
$T_{12,3}$ & $(12,7)$ & $(1, 1, 1, 0, 1, 2, 2, 0, 2, 0, 0, 0, 2)$\\
$T_{13,3}$ & $(13,8)$ & $(1, 1, 2, 0, 1, 2, 1, 0, 2, 0, 0, 2, 1)$\\
$T_{14,3}$ & $(14,8)$ & $(1, 1, 2, 0, 2, 2, 2, 0, 2, 0, 0, 0, 2)$\\
$T_{15,3}$ & $(15,9)$ & $(1, 1, 2, 0, 1, 2, 2, 0, 2, 0, 0, 2, 2)$\\
\noalign{\hrule height0.8pt}
 \end{tabular}
}
\end{center}
\end{table}

\begin{table}[thb]
\caption{$g_3(n,3)$}
\label{table:G-bound-3-3}
\begin{center}
{\small
\begin{tabular}{c|c||c|c||c|c}
\noalign{\hrule height0.8pt}
$n$ & $g_3(n,3)$ & $n$ & $g_3(n,3)$ & $n$ & $g_3(n,3)$ \\
\hline
$13s$ & $9s$ & $13s+5$ & $9s+3$ & $13s+10$ & $9s+6$ \\
$13s+1$ & $9s$ & $13s+6$ & $9s+3$ & $13s+11$ & $9s+7$ \\
$13s+2$ & $9s$ & $13s+7$ & $9s+4$ & $13s+12$ & $9s+8$ \\
$13s+3$ & $9s+1$ & $13s+8$ & $9s+5$ &&\\
$13s+4$ & $9s+2$ & $13s+9$ & $9s+6$ &&\\
\noalign{\hrule height0.8pt}
\end{tabular}
}
\end{center}
\end{table}

By Lemma~\ref{lem:n0},
there is no ternary LCD $[13s,3,9s]$ code
for every positive integer $s$.
We apply Theorem~\ref{thm:main} (ii)
to the case $(q,k)=(3,3)$.
For $n \ge 5$,
write $n=13s+t$, where $s \in \ZZ_{\ge 0}$ and $t \in 
\{1,5,8,9,11,12\}$.
Suppose that $d(s,t)=g_3(13s+t,3)$.
Let $r=r_{3,13s+t,3,d(s,t)}$
be the integer defined in~\eqref{eq:r}.
For each $13s+t$, we list $d(s,t)$ and $r$
in Table~\ref{Tab:3-3-0}.
Then $d(s,t)$ is written as
$9s+\alpha(t)$, where $\alpha(t)$ is a constant depending on only $t$.
Since the minimum weight satisfies 
the assumption~\eqref{eq:as} in Theorem~\ref{thm:main2},
we have the following:

\begin{prop}\label{prop:3-3-0}
If there is no ternary LCD $[3r,3,2r]$ code
with dual distance $d^\perp \ge 2$,
then there is no ternary LCD $[13s+t,3,d(s,t)]$ code
for every nonnegative integer $s$.
\end{prop}

\begin{table}[thbp]
\caption{$d(s,t)$ and $r$ in Proposition~\ref{prop:3-3-0}}
\label{Tab:3-3-0}
\begin{center}
{\small
\begin{tabular}{c|c|c||c|c|c}
\noalign{\hrule height0.8pt}
  $n$     &$d(s,t)$ & $r$ &  $n$     &$d(s,t)$ & $r$ \\
\hline
  $13s+1$ &$9s$  &  $9$ &  $13s+9$  &$9s+6$ & $3$ \\
  $13s+5$ &$9s+3$&  $6$ &  $13s+11$ &$9s+7$ & $8$ \\
  $13s+8$ &$9s+5$&  $7$ &  $13s+12$ &$9s+8$ & $4$ \\
  \noalign{\hrule height0.8pt}
\end{tabular}
}
\end{center}
\end{table}

It is known that there is no ternary LCD $[9,3,6]$ code~\cite[Table~4]{AH-C}.
By the method given in Section~\ref{sec:Cmethod},
our computer search shows that there is no ternary LCD $[3r,3,2r]$ code
with dual distance $d^\perp \ge 2$ for $r \in \{4,6,7,8,9\}$. 
By Proposition~\ref{prop:3-3-0}, we have the following:

\begin{prop}
Suppose that $n \ge 3$.  Then
\[
d_3(n,3)=
\begin{cases}
\left\lfloor \frac{9n}{13} \right\rfloor & \text{if } n \equiv 4,7,10 
\pmod{13}, \\
\left\lfloor \frac{9n}{13} \right\rfloor-1 & \text{if } n \equiv 
0,1,2,3,5,6,8,9,11,12 \pmod{13}.
\end{cases}
\]
\end{prop}

\subsubsection{Classification of ternary optimal LCD codes of dimension 3}

We apply Theorem~\ref{thm:main} (i) to the case $(q,k)=(3,3)$.
For $n \ge 3$,
write $n=13s+t$, where $s \in \ZZ_{\ge 0}$ and 
$t \in \{0,1,\ldots,12\}$.
Let $r=r_{3,13s+t,3,d_3(13s+t,3)}$ and
$s'=s'_{3,13s+t,3,d_3(13s+t,3)}$ be the integers defined
in~\eqref{eq:r} and~\eqref{eq:s0}, respectively.
For each $13s+t$, we list $d_3(n,3)$, $s'$ and $r$ in Table~\ref{Tab:3-3}.
Then $d_3(13s+t,3)$ is written as
$9s+\alpha(t)$, where $\alpha(t)$ is a constant depending on only $t$.
Since the minimum weight satisfies 
the assumption~\eqref{eq:as} in Theorem~\ref{thm:main2},
we have the following:

\begin{prop}\label{prop:3-3}
There is a one-to-one correspondence between
equivalence classes of ternary LCD
$[3r,2,2r]$ codes with dual distances $d^\perp \ge 2$
and
equivalence classes of ternary LCD $[13s+t,3,d_3(13s+t,3)]$ codes
with dual distances $d^\perp \ge 2$
for every integer $s \ge s'$.
\end{prop}

\begin{table}[thbp]
\caption{Ternary optimal LCD $[n,3]$ codes with  dual distances $d^\perp\ge 2$}
\label{Tab:3-3}
\begin{center}
{\small
\begin{tabular}{c|c|c|c|c}
\noalign{\hrule height0.8pt}
 $n$ &$d_3(n,3)$& $s'$ & $r$ & $N_3(n,3)$\\
\hline
 $13s   $&$9s-1$&$4$ &  13&3\\
 $13s+1 $&$9s-1$&$6$ &  22&144\\
 $13s+2 $&$9s  $&$5$ &  18&15\\
 $13s+3 $&$9s+1$&$4$ &  14&4\\
 $13s+4 $&$9s+2$&$3$ &  10&1\\
 $13s+5 $&$9s+2$&$5$ &  19&45\\
 $13s+6 $&$9s+3$&$4$ &  15&4\\
 $13s+7 $&$9s+4$&$3$ &  11&1\\
 $13s+8 $&$9s+4$&$5$ &  20&54\\
 $13s+9 $&$9s+5$&$4$ &  16&13\\
 $13s+10$&$9s+6$&$3$ &  12&1\\
 $13s+11$&$9s+6$&$5$ &  21&54\\
 $13s+12$&$9s+7$&$4$ &  17&15\\
\noalign{\hrule height0.8pt}
\end{tabular}
}
\end{center}
\end{table}

\begin{table}[thbp]
\caption{$N_3(n,3)$}
\label{Tab:3-3-2}
\begin{center}
{\small
\begin{tabular}{c|c||c|c||c|c||c|c}
\noalign{\hrule height0.8pt}
$n$ & $N_3(n,3)$& $n$ & $N_3(n,3)$&
$n$ & $N_3(n,3)$& $n$ & $N_3(n,3)$\\
 \hline
$ 4$ &  1&$17$ &  1&$30$ &  1&$44$ & 45\\
$ 5$ &  2&$18$ & 26&$31$ & 45&$45$ &  4\\
$ 6$ &  2&$19$ &  4&$32$ &  4&$47$ & 54\\
$ 7$ &  1&$20$ &  1&$33$ &  1&$48$ & 13\\
$ 8$ &  7&$21$ & 41&$34$ & 54&$50$ & 54\\
$ 9$ &  3&$22$ & 13&$35$ & 13&$51$ & 15\\
$10$ &  1&$23$ &  1&$36$ &  1&$53$ &144\\
$11$ & 12&$24$ & 46&$37$ & 54&$54$ & 15\\
$12$ &  8&$25$ & 15&$38$ & 15&$57$ & 45\\
$13$ &  3&$26$ &  3&$39$ &  3&$60$ & 54\\
$14$ & 39&$27$ &110&$40$ &144&$63$ & 54\\
$15$ & 10&$28$ & 15&$41$ & 15&$66$ &144\\
$16$ &  4&$29$ &  4&$42$ &  4& &\\
 \noalign{\hrule height0.8pt}
\end{tabular}
}
\end{center}
\end{table}

\begin{itemize}
\item Case $s \ge s'$:
By the method given in Section~\ref{sec:Cmethod},
our computer search completes a classification of ternary
LCD $[3r,3,2r]$ codes with dual distances
$d^\perp \ge 2$ for $r\in R$, where
\[
R=\{
 10,
 11,
 12,
 13,
 14,
 15,
 16,
 17,
 18,
 19,
 20,
 21,
 22\}.
\]
The numbers $N_3(3r,3)$ of inequivalent
ternary LCD $[3r,3,2r]$ codes with dual distances
$d^\perp \ge 2$ are listed in Table~\ref{Tab:3-3-2}.
These codes $T_{3r,3,i}$ $(i\in \{1,2,\ldots,N_3(3r,3)\})$
are presented by codes $C_{3,3}(m)$
with generator matrices $G_{3,3}(m)$ of form~\eqref{eq:Gqkm}.
For $r \in \{10,11,12\}$, 
in order to display the codes $T_{3r,3,i}$,
the vectors $m$
are listed in Table~\ref{Tab:3-3-3}.
For all the parameters $r$, the vectors $m$
are available at
\url{http://www.math.is.tohoku.ac.jp/~mharada/AHS}.

\begin{table}[thb]
\caption{Ternary LCD $[n,3,d]$ codes $T_{n,3,i}$}
\label{Tab:3-3-3}
\begin{center}
{\small
\begin{tabular}{c|c|c}
\noalign{\hrule height0.8pt}
$C_{3,3}(m)$ & $(3r,2r)$ & $m$ \\\hline
$T_{30,3,1}$&$(30,20)$& $( 3, 3, 2, 2, 3, 2, 2, 3, 2, 2, 2, 2, 2)$\\
$T_{33,3,1}$&$(33,22)$& $( 3, 3, 3, 2, 3, 3, 3, 3, 2, 2, 2, 2, 2)$\\
$T_{36,3,1}$&$(36,24)$& $( 3, 3, 3, 3, 3, 3, 3, 3, 2, 3, 3, 2, 2)$\\
 \hline
$C_{3,3}(m)$ & $(n,d)$ & $m$ \\\hline
$T_{17,3,1}$&$(17,11)$&$(2,2,1,1,2,1,1,2,1,1,1,1,1)$\\
$T_{20,3,1}$&$(20,13)$&$(2,2,2,1,2,2,2,2,1,1,1,1,1)$\\
$T_{23,3,1}$&$(23,15)$&$(2,2,2,2,2,2,2,2,1,2,2,1,1)$\\
 \noalign{\hrule height0.8pt}
 \end{tabular}
}
\end{center}
\end{table}

By Proposition~\ref{prop:3-3},
we have a classification of ternary optimal
LCD $[13s+t,3]$ codes  with dual distances $d^\perp \ge 2$ for
      $s \ge s'$.
      
\item Case $s < s'$:
A classification of ternary LCD codes was done in~\cite[Table~4]{AH-C}
for lengths up to $10$.
For $n \le 10$,
we calculated the numbers $N_3(n,3)$ of inequivalent
ternary optimal LCD $[n,3]$ codes with dual distances $d^\perp \ge 2$,
where the numbers $N_3(n,3)$ are listed in Table~\ref{Tab:3-3-2}.
By the method given in Section~\ref{sec:Cmethod},
our computer search completes a classification of ternary optimal
LCD $[n,3]$ codes with dual distances
$d^\perp \ge 2$ for $n \in N \setminus \{3r \mid r \in R\}$, where
\[
N=
\{11,12,\ldots,42,
44,
45,
47,
48,
50,
51,
53,
54,
57,
60,
63,
66
\}.\]
The numbers $N_3(n,3)$ are listed in Table~\ref{Tab:3-3-2}.
These codes $T_{n,3,i}$ $(i \in \{1,2,\ldots,N_3(n,3)\})$
are presented by codes $C_{3,3}(m)$
with generator matrices $G_{3,3}(m)$ of form~\eqref{eq:Gqkm}.
For $(n,d) \in \{(17,11),(20,13),(23,15)\}$,
in order to display the codes $T_{n,3,i}$,
the vectors $m$
are listed in Table~\ref{Tab:3-3-3}.
For all the lengths $n$, the vectors $m$
are available at
\url{http://www.math.is.tohoku.ac.jp/~mharada/AHS}.
\end{itemize}
Therefore, we have the following:

\begin{prop}\label{prop:3-3-2}
\begin{enumerate}
\item 
There are $3$ inequivalent ternary optimal LCD
      $[13s,3]$ codes with dual distances $d^\perp \ge 2$
      for every integer $s \ge 1$.
\item 
There are $144$ inequivalent ternary optimal LCD
      $[13s+1,3]$ codes with dual distances $d^\perp \ge 2$
      for every integer $s \ge 3$.
\item 
Suppose that $t \in \{2,12\}$.
Then there are $15$ inequivalent ternary optimal LCD
      $[13s+t,3]$ codes with dual distances $d^\perp \ge 2$
      for every integer $s \ge 2$
if $t=2$ and $s \ge 1$ if $t=12$.
      
\item Suppose that $t \in \{3,6\}$.
Then
there are $4$ inequivalent ternary optimal LCD
      $[13s+t,3]$ codes with dual distances $d^\perp \ge 2$
      for every integer $s \ge 1$.

\item Suppose that $t \in \{4,7,10\}$.
Then
there is a unique ternary optimal LCD
$[13s+t,3]$ code with dual distance $d^\perp \ge 2$, up to equivalence,
for every integer $s \ge 0$.

\item 
There are $45$ inequivalent ternary optimal LCD
      $[13s+5,3]$ codes with dual distances $d^\perp \ge 2$
      for every integer $s \ge 2$.

\item Suppose that $t \in \{8,11\}$.
Then there are $54$ inequivalent ternary optimal LCD
      $[13s+t,3]$ codes with dual distances $d^\perp \ge 2$
      for every integer $s \ge 2$.

\item 
There are $13$ inequivalent ternary optimal LCD
      $[13s+9,3]$ codes with dual distances $d^\perp \ge 2$
      for every integer $s \ge 1$.
\end{enumerate}
\end{prop}

\begin{rem}
If there is a ternary $[n,n-3,d]$ code with $d \ge 3$, then
$n \le 13$ by the sphere-packing bound. 
Let $T_{n,3,1}$ denote the unique ternary optimal LCD
$[n,3]$ code for $n \in\{4,7,10\}$.
The dual distances of all the codes except 
$T_{n,3,1}$ $(n \in\{4,7,10\})$ in the above proposition
are exactly $2$.
The codes $T_{n,3,1}$ $(n \in\{4,7,10\})$ 
have dual distances $3$.
\end{rem}

By Lemma~\ref{lem:class}, the above proposition
completes a classification of ternary optimal LCD codes of
dimension $3$.

\bigskip
\noindent
{\bf Acknowledgments.}
The authors would like to thank Tatsuya Maruta for his useful comments.
This work was supported by JSPS KAKENHI Grant Number 19H01802.



\begin{thebibliography}{99}

\bibitem{AH-C}M. Araya and M. Harada, 
On the classification of linear complementary dual codes,
{\sl Discrete Math.}
{\bf 342} (2019), 270--278. 


\bibitem{AH} M. Araya and M. Harada,
On the minimum weights of binary linear complementary dual codes,
{\sl Cryptogr.\ Commun.}
{\bf 12} (2020), 285--300.

\bibitem{AHS} M. Araya, M. Harada and K. Saito,
Quaternary Hermitian linear complementary dual codes,
{\sl IEEE Trans.\ Inform.\ Theory}
{\bf 66} (2020), 2751--2759.
	

\bibitem{Bonisoli}A. Bonisoli, 
Every equidistant linear code is a sequence of dual Hamming codes,
{\sl Ars Combin.} {\bf 18}  (1984), 181--186. 	

	
\bibitem{Magma}W. Bosma, J. Cannon and C. Playoust,
The Magma algebra system I: The user language,
{\sl J. Symbolic Comput.}
{\bf 24} (1997), 235--265.

\bibitem{CL} P.J. Cameron and J.H. van Lint,
Designs, Graphs, Codes and Their Links,
Cambridge University Press, Cambridge, 1991.

\bibitem{CG}
C. Carlet and S. Guilley,
Complementary dual codes for counter-measures to side-channel attacks,
{\sl Adv.\ Math.\ Commun.}
{\bf 10}  (2016),  131--150.
	
\bibitem{CMTQ}
C. Carlet, S. Mesnager, C. Tang and Y. Qi, 
New characterization and parametrization of LCD codes,
{\sl IEEE Trans.\ Inform.\ Theory}
{\bf 65} (2019), 39--49.
	
\bibitem{CMTQP}
C. Carlet, S. Mesnager, C. Tang, Y. Qi and R. Pellikaan,
Linear codes over $\FF_q$ are equivalent to LCD codes for $q >3$,
{\sl IEEE\ Trans.\ Inform.\ Theory}
{\bf 64}  (2018),  3010--3017.




 \bibitem{FLFR}Q. Fu, R. Li, F. Fu and Y. Rao,
On the construction of binary optimal LCD codes with short length,
{\sl Internat.\ J.\ Found.\ Comput.\ Sci.}
{\bf 30} (2019), 1237--1245. 

\bibitem{GKLRW}
L. Galvez, J.-L. Kim, N. Lee, Y.G. Roe and B.-S. Won,
Some bounds on binary LCD codes,
{\sl Cryptogr.\ Commun.}
{\bf 10} (2018), 719--728.

\bibitem{HS} M. Harada and K. Saito,
Binary linear complementary dual codes,
{\sl Cryptogr.\ Commun.}
{\bf 11} (2019), 677--696.


\bibitem{HS2}
M. Harada and K. Saito,
Remark on subcodes of linear complementary dual codes,
{\sl Inform.\ Process.\ Lett.} 
{\bf 159} (2020), 105963 (3 pp.).

\bibitem{HP} W.C. Huffman and V. Pless,
Fundamentals of Error-Correcting Codes,
Cambridge University Press, Cambridge,
(2003).

	
\bibitem{LN} E.R. Lina, Jr.\ and E.G. Nocon,
On the construction of some LCD codes over finite fields,
{\sl Manila J. Science}
{\bf 9} (2016), 67--82.
	
\bibitem{LLGF}L. Lu, R. Li, L. Guo and Q. Fu, 
Maximal entanglement entanglement-assisted quantum codes 
constructed from linear codes,
{\sl Quantum Inf.\ Process.}
{\bf 14} (2015), 165--182.

\bibitem{Maruta} T. Maruta, 
On the achievement of the Griesmer bound,
{\sl Des.\ Codes Cryptogr.} {\bf 12}  (1997),  83--87. 
	
\bibitem{Massey}J.L. Massey, 
Linear codes with complementary duals,
{\sl Discrete Math.}
{\bf 106/107} (1992), 337--342.

\bibitem{nauty} B.D. McKay and A. Piperno,
Practical graph isomorphism, II, 
{\sl J. Symbolic Comput.}
{\bf 60} (2014), 94--112.
	
\bibitem{PZK}B. Pang, S. Zhu and X. Kai,
Some new bounds on LCD codes over finite fields,
{\sl Cryptogr.\ Commun.}
{\bf 12} (2020), 743--755.
	

 \bibitem{ntl} V. Shoup,
NTL: A Library for doing Number Theory, 
Available online at \url{http://www.shoup.net/ntl/}.
	

\end{thebibliography}
\end{document}